\documentclass[reqno,11pt]{amsart}
\usepackage[T1]{fontenc} 
\usepackage[utf8]{inputenc} 
\usepackage{mathtools}
\usepackage{amssymb}
\usepackage{amsthm}
\usepackage{bbm}
\usepackage{mathrsfs}
\usepackage[english]{babel} 
\usepackage{enumitem}
\usepackage{dsfont}
\usepackage{url}
\usepackage{hyperref}
\usepackage{tikz}
\usetikzlibrary{trees}

\hypersetup{linktocpage=true, colorlinks=true, breaklinks=true, linkcolor=blue, menucolor=black, urlcolor=black,citecolor=blue,filecolor=green}

\theoremstyle{plain}
\newtheorem{theorem}{Theorem}[section]
\newtheorem{proposition}[theorem]{Proposition}

\newtheorem{lemma}[theorem]{Lemma}

\theoremstyle{definition}

\newtheorem{remark}[theorem]{Remark}

\newcommand{\floor}[1]{\left\lfloor #1 \right\rfloor}
\def\eps{\varepsilon}

\numberwithin{equation}{section}

\newcommand{\nn}{\nonumber}

\DeclareMathSymbol{\leqslant}{\mathalpha}{AMSa}{"36}
\DeclareMathSymbol{\geqslant}{\mathalpha}{AMSa}{"3E}
\DeclareMathSymbol{\doteqdot}{\mathalpha}{AMSa}{"2B}
\DeclareMathSymbol{\circlearrowright}{\mathalpha}{AMSa}{"08}
\DeclareMathSymbol{\subsetneq}{\mathalpha}{AMSb}{"28}
\DeclareMathSymbol{\supsetneq}{\mathalpha}{AMSb}{"29}
\renewcommand{\leq}{\;\leqslant\;}
\renewcommand{\geq}{\;\geqslant\;}

\newcommand{\dd}{{\rm d}}
\newcommand{\e}[1]{\,{\rm e}^{#1}\,}

\DeclareMathOperator*{\supp}{\mathrm{supp}}

\newcommand{\upchi}{\raise 2pt \hbox{$\chi$}}

\newcommand{\caA}{{\mathcal A}}

\newcommand{\caF}{{\mathcal F}}

\newcommand{\caH}{{\mathcal H}}

\newcommand{\bbE}{{\mathbb E}}

\newcommand{\bbN}{{\mathbb N}}

\newcommand{\bbP}{{\mathbb P}}

\newcommand{\bbR}{{\mathbb R}}

\newcommand{\bbZ}{{\mathbb Z}}

\newcommand{\de}{{\mathrm{d}}}
\newcommand{\one}{{\mathbbm{1}}}
\newcommand{\bbPh}{\hat{\bbP}}
\newcommand{\bbPw}{\Tilde{\bbP}}

\def\lt{\left}
\def\rt{\right}

\def\la{\langle}
\def\ra{\rangle}
\usepackage{comment}

\title{Sub-diffusive regimes for long range self-interacting path measures}
\author{Volker Betz}
\address{Department of Mathematics, TU Darmstadt, Germany}
\email{betz@mathematik.tu-darmstadt.de}

\author{Tobias Schmidt}
\address{Department of Mathematics, TU Darmstadt, Germany}
\email{tobias.schmidt@tu-darmstadt.de}

\author{Mark Sellke}
\address{Department of Statistics, Harvard University, Cambridge, MA, USA}
\email{msellke@fas.harvard.edu}
\date{\today}
\begin{document}

\begin{abstract}
    We study Brownian motion perturbed by a long range self-interaction. We provide variance bounds in terms of the spatial interaction strength and the order of time decay. 
\end{abstract}
\maketitle
\section{Introduction and main results}

\subsection{Motivation and Results}
Recently, there has been a lot of activity studying self-interacting path measures of the form
\begin{equation}
    \label{equ:general_qft_measure}
    \bbP_{\mathbf{self}}(\de x) \propto \exp \lt( \alpha \int_0 ^T \de t \int_0 ^T \de s W(|x_t - x_s|,|t-s|) \rt) \bbP_0(\de x),
\end{equation}
where $\bbP_0$ is the path measure of $d$-dimensional Brownian motion, and 
$(x,t) \mapsto W(x,t)$ is suitably regular, unimodal in $x$ for all $t$, and maximal at $x=0$.  See for example 
\cite{Se24,BePo22,BePo23,BeSchSe25,BaMuSeVa23,BeSchSe24,MuVa19}. An important source of interest 
for these models is their relation to quantum field theory, specifically polaron models, on which we will give some details below. For now, let us focus on the probabilistic object 
at hand. 

The probability measure \eqref{equ:general_qft_measure} is reminiscent of a Gibbs measure 
as known from statistical mechanics. There is a short (infinitesimal) range potential 
due to the Brownian motion, and a pair potential $W$. By the spatial unimodality assumption on $W$, the interaction is 
attractive, and so it is natural to expect the paths to fluctuate less between times 
$0$ and $T$ under $\bbP_{\rm self}$ than they would under Brownian motion. A convenient 
quantity to measure such effects is the behaviour of the mean square displacement 
\[
s_T = \bbE_{\rm self}(|x_T|^2)
\]
for large $T$.

There are two possible cases. The first one, which is also the one that is best understood, 
is when the decay of $W$ in its second (temporal) argument is fast. Then, it has been shown in fair generality 
\cite{Sp86, BeSp05} that 
$\sigma^2 := \liminf_{T \to \infty}\frac{s_T}{T} > 0$, meaning that the paths still
behave diffusively. The diffusion constant $\sigma^2$ is expected to be strictly less than 
$1$ for a large class of attractive interactions $W$, reflecting the fact that the 
fluctuations are indeed reduced by the presence of the potential. Since $\sigma^2$ has 
intimate connections with the effective mass of polaron models (more on this below), it has 
been investigated a lot. We now have good upper bounds for $\sigma^2$ under a variety of 
conditions, either by direct investigation of the path measures 
\eqref{equ:general_qft_measure} (see \cite{BaMuSeVa23,BePo23,Se24,BeSchSe25}), or by analytic 
means via the connection to polaron models \cite{BrSe23,BrSe24,Br24}). 

The second possibility is that $s_T$ grows sublinearly as $T\to\infty$.
In other words, the paths behave sub-diffusively between times $0$ and $T$ under 
$\bbP_{\rm self}$, or even stay localized in the extreme case that $s_T$ 
stays bounded as $T \to \infty$. This regime is much less understood, and it is the one 
that we explore in the  present paper. Our first result is as follows: let  
\begin{equation}
    \label{equ:most_basic_measure}
    \bbPh_{\alpha,T,\xi} (\de x) \propto \exp \lt( - \alpha \int_0 ^T \dd s \int_0 ^T \dd t \, \frac{f(x_t - x_s)}{1+ |t-s|^\xi} \rt) \bbP(\de x),
\end{equation} 
where   
\begin{enumerate}
        \item $f:\bbR^d\to\bbR_+$ is radially symmetric.
        \item There exists $\zeta >0$ such that $\bbR^d \ni z \mapsto  f(z) - \zeta |z|^2$ is quasi-convex.
\end{enumerate}
\begin{theorem}
    \label{thm:main_thm_1}
    Let $ 0 \le \xi < 3$ and $\alpha > 0$. Then, there exists a constant $C>0$ such that for all measures as in \eqref{equ:most_basic_measure}:
    \[
    \bbE^{\bbPh_{\alpha,T,\xi}}[|x_T|^2] \le 16 \lt( C \min\{\alpha^{-1/2},1\} + \alpha^{-1} \max \lt\{  \zeta^{-1} \log(T)^2 T^{\xi-2}, 1 \rt\} \rt).
    \]
    For $\xi \in (0,1)$ we even find
    $$\lim\limits_{T \to \infty}\bbE^{\bbPh_{\alpha,T,\xi}}[|x_T|^2] = 0.$$
\end{theorem}
We thus obtain sub-diffusive behaviour in the range $1 < \xi < 3$, and localization in the 
range $0 \leq \xi < 2$; indeed, we show that the mean square displacement converges to zero 
in case $\xi \in (0,1)$.  From the point of view of statistical mechanics, this can be expected 
due to the following heuristic: 
for $\xi < 1$, the interaction is so long range that 
$\int_0^T \dd s \int_0^T \dd t \, (1 + |s-t|^\xi)^{-1}$ grows faster than linearly in $T$. 
Since the probability of a Brownian path to stay inside a compact subset of $\bbR^d$ up to 
time $T$ decays (not faster than) exponentially in $T$ (see \cite{AKS25} for some very 
precise asymptotics), this means that the energy cost of 
being away from the minimum of $f$ (which is at $x \equiv 0$) for an amount of time proportional to 
$T$ outweighs the entropy cost of staying in this area. Hence one expects typical paths to remain close to zero in most of $[0,T]$. 

For the case where $f(x) = |x|^2$, the measure $\bbPh_{\alpha, T,\xi}$ is Gaussian. In this 
case, formula (3.13) in \cite{Sp87} suggests that 
\[
s_T \sim \int_0^\infty \dd k \frac{1 - \cos(2 \pi T k)}{4 \pi^2 k^2 + \alpha (\hat g(0) - \hat g(k))}
\]
for large $T$, where $g(r) = \frac{1}{1 + |r|^\xi}$ with $\xi > 1$, and $\hat g$ 
denotes the Fourier transform. No justification of this formula is given in \cite{Sp87}, 
but the result can be reproduced with a Gaussian calculation for the case where Brownian
paths are considered on the time interval $[0,T]$ with periodic boundary conditions, 
and the interaction is modified to reflect these boundary conditions as well. We provide the details in the Appendix. These 
boundary modifications are not necessarily harmless in view of the long range potential, but it is reasonable to expect that at least the leading order behavior in $T$ is correct. When 
$g(x)$ decays like $|x|^{-\xi}$ at infinity, then $\hat g(0) - \hat g(k) \sim |k|^{\xi - 1}$
as $|k| \to 0$, since $\hat g$ has a maximum at $k=0$. By re-scaling, it follows that for 
$1 < \xi \leq 3$, 
\begin{equation}
    \label{equ:gaussian_calculation_pred}
    s_T \sim T \int_0^\infty \dd u \,  \frac{1 - \cos(2 \pi u)}{4 \pi^2 u^2 + \alpha T^2
(\hat g(0) - \hat g((u/T))} \sim T^{\xi - 2}. 
\end{equation}
This means that sub-diffusive behaviour can be expected for $2<\xi<3$, and localization
for $1 < \xi < 2$. 
This is also what we verify.

The global domination assumption (2) above is critical to our result, as we rely on the 
Gaussian correlation inequality (GCI). When this assumption fails, things become much more 
difficult. One would still anticipate a regime of subdiffusivity and one of localization at least
when $f$ diverges at infinity, and possibly also for bounded $f$ that make $W$ attractive. This is because, heuristically, slower time decay can compensate for the weaker 'spatial penalty' as compared to the square case.
In this work we focus on the case where we have at least some divergence at infinity. For any path measure $\mu$, $\gamma > 0$, $\xi > 0$ and $\alpha > 0$, let
\begin{equation}
    \label{equ:gamma_epsilon_bm}
    \Tilde{\mu}_{\alpha,T,\gamma,\xi}(\de x) \propto \exp \lt( -\alpha \int_0 ^T \de t \int_0 ^T \de s \frac{| x_t - x_s|^{\gamma }}{1+|t-s|^{\xi}} \rt) \mu(\de x),
\end{equation}
and
$$\hat{\mu}_{\alpha,T,f,\xi}(\de x) \propto \exp \lt( -\alpha \int_0 ^T \de t \int_0 ^T \de s \frac{f(x_t - x_s)}{1+|t-s|^{\xi}} \rt) \mu(\de x).$$
We will write $Z_\mu$ for the corresponding partition sum of $\Tilde{\mu}_{\alpha,T,\gamma,\xi}$ in case it is clear from the context what all occurring parameters are.
It will happen that $\mu$ has an upper and lower index itself. We then write $(\tilde{\mu}^a _b)_{\alpha,T,\gamma,\xi}$ instead of  $\widetilde{(\mu^a _b)}_{\alpha,T,\gamma,\xi}$ for aesthetic purposes. If $\mu$ just has an upper index, we similarly write $\Tilde{\mu}^a_{\alpha,T,\gamma,\xi}$ instead of $\widetilde{(\mu^a)}_{\alpha,T,\gamma,\xi}$. In case $\gamma =2$ we abbreviate $\Tilde{\mu}_{\alpha,T,2,\xi} = \Tilde{\mu}_{\alpha,T,\xi}$. 
We then have
\begin{theorem}
    \label{thm:sub_diffusive_regime}
     Let $0 <\gamma < 2$. 
   Suppose that $f:\bbR^d\to\bbR_+$ is radially symmetric, and there exists $\zeta >0$ such that 
   \begin{equation}
       \nn
       \bbR^d \ni z \mapsto  f(z) - \zeta |z|^{\gamma}
   \end{equation}
    is quasi-convex.
    Then, there exists a constant $C= C(\zeta,\gamma)> 0$ such that
    \begin{enumerate}
        \item for $\xi \in (1+\gamma/2,2+\gamma/2)$ we have
        $$\bbE^{\bbPh_{\alpha,T,f,\xi}}[|x_T|^2] \le C \log(T)^4T^{\xi-1-\gamma/2},$$
        \item for $\xi \in (\gamma/2,1+\gamma/2)$ we have
        $$\limsup\limits_{T \to \infty}\bbE^{\bbPh_{\alpha,T,f,\xi}}[|x_T|^2] \le C,$$
        \item for $\xi \in (0,\gamma/2)$ we have
        $$\limsup\limits_{T \to \infty}\bbE^{\bbPh_{\alpha,T,f,\xi}}[|x_T|^2] = 0.$$
    \end{enumerate}
\end{theorem}
The leading order in the regime $\xi \in (2,2+\gamma/2)$ is suboptimal. The next Theorem improves the rate to $2(\xi-2)/\gamma$.
\begin{theorem}
    \label{thm:main_thm_2}
     Let $0 <\gamma < 2$ and $\xi \in (2,\infty)$. 
   Suppose that $f:\bbR^d\to\bbR_+$ is radially symmetric, and there exists $\zeta >0$ such that 
   \begin{equation}
       \label{equ:qc_condition}
       \bbR^d \ni z \mapsto  f(z) - \zeta |z|^{\gamma}
   \end{equation}
    is quasi-convex.
    Then, there exists a constant $C = C(\zeta,\gamma)>0$ such that for all $\alpha >0$ 
    $$ \bbE^{\bbPh_{\alpha,T,f,\xi}}[|x_T|^2] \le C T^{\frac{2}{\gamma}(\xi-2)}\log(T)^2.$$
\end{theorem}
The above bound is only interesting for $\xi \in (2, 2 + \gamma/2)$. Otherwise, a simple application of (GCI) leads to the simpler bound $T$.

Finally, we are interested in letting $\alpha$ tend to infinity together with $T$. Here, we can prove that a transition exists as soon as $\xi <2$: for any $\gamma \in (0,2)$, fluctuations are at most logarithmic.
\begin{theorem}
    \label{thm:main_thm_3}
    Let $0 <\gamma < 2$ and $\xi \in (0,2)$. 
   Suppose that $f:\bbR^d\to\bbR_+$ is radially symmetric, and there exists $\zeta >0$ such that 
   \begin{equation}
       \nn
       \bbR^d \ni z \mapsto  f(z) - \zeta |z|^{\gamma}
   \end{equation}
    is quasi-convex.
    Then, there exists a constant $C = C(\zeta,\gamma)>0$ such that for all $\alpha \ge  C\log(T)^{1/2}$ 
    $$ \bbE^{\bbPh_{\alpha,T,f,\xi}}[|x_T|^2] \le C\log(T)^{5/2}.$$
\end{theorem}
Our technique inherently does not provide lower bounds on the variance, so we can only talk about the different regimes in terms of upper bounds.
While logarithmic divergence occurs in the Gaussian case only for $\xi = 2$, we find that there exists an entire interval of $\xi$ for the general $\gamma \in (0,2)$ case where logarithmic divergence holds if one increases $\alpha$ only mildly with $T$. The fact that we obtain only logarithmic divergence could be an artifact of our approach: we always sum over $\log(T)$ factors. This is necessary, because we rely on a dyadic decomposition of the path. Therefore, it might be that the logarithmic divergence can be improved to actual boundedness. 

The provided bounds are the natural generalizations of the Gaussian case, which can be recovered by (formally) setting $\gamma = 2$. To summarize, we find that the regime of sub-diffusivity shrinks for $\gamma <2$: compared to the Gaussian case, sub-diffusive behaviour only occurs for $\xi \in (2,2+\gamma/2)$. Moreover, the interval $(\gamma/2,1+\gamma/2)$ becomes the regime with bounded variance, whereas the interval $(1+\gamma/2,2+\gamma/2)$ becomes the sub-diffusive regime. Finally, we verify the convergence of variance to $0$ only for the parameter range $\xi \in (0,\gamma/2)$. Therefore, all regimes shift by $\gamma/2$ to the left. Interestingly, by letting $\alpha$ only grow logarithmically in $T$, much stronger bounds on the concentration can be obtained in the regime $\xi \in (1+\gamma/2,2)$. The following Figure \ref{fig:placeholder} summarizes our findings.
\begin{figure}
    \centering
    \begin{tikzpicture}[scale=1.9,>=stealth]

  \draw[->,thick] (0,0) -- (4.2,0) node[right] {$\xi$};
  \draw[->,thick] (0,0) -- (0,2.3) node[above] {$\gamma$};

  \foreach \x in {0,0.5,1,1.5,2,2.5,3,3.5,4}{
    \draw (\x,0) -- (\x,-0.05);
    \node[below] at (\x,0) {\scriptsize \x};
  }
  \foreach \g in {1,2}{
    \draw (-0.05,\g) -- (0.05,\g);
    \node[left] at (0,\g) {\scriptsize \g};
  }

  \foreach \g in {0.01,0.02,...,2.0}{
    
    \draw[green!70!black,opacity=0.55,line width=0.8pt]
      (0,\g) -- ({0.5*\g},\g);
    
    \draw[brown!80!black,opacity=0.55,line width=0.8pt]
      ({0.5*\g},\g) -- ({1+0.5*\g},\g);
    
    \draw[blue!70!black,opacity=0.55,line width=0.8pt]
      ({1+0.5*\g},\g) -- (2,\g);
   
    \draw[red!70!black,opacity=0.55,line width=0.8pt]
      (2,\g) -- ({2+0.5*\g},\g);
   
    \draw[magenta!80!black,opacity=0.55,line width=0.8pt]
      ({2+0.5*\g},\g) -- (4,\g);
  }

  \draw[dotted, thick] (3,0) -- (3,2.0);

  \begin{scope}[shift={(4.5,1.1)}]
    \matrix[anchor=west, column sep=5pt, row sep=1pt, scale=0.6]{
      \draw[green!70!black, line width=1pt] (0,0) -- (0.5,0); &
        \node[align=left] {\tiny Variance collapse (Theorem \ref{thm:sub_diffusive_regime})}; \\
      \draw[brown!80!black, line width=1pt] (0,0) -- (0.5,0); &
        \node[align=left] {\tiny Bounded variance (Theorem \ref{thm:sub_diffusive_regime})}; \\
      \draw[blue!70!black, line width=1pt] (0,0) -- (0.5,0); &
        \node[align=left] {\tiny Logarithmic divergence ($\alpha = O(\log(T))$)/ \\ \tiny Subdiffusive divergence (Theorem \ref{thm:main_thm_3}/\ref{thm:sub_diffusive_regime})  }; \\
      \draw[red!70!black, line width=1pt] (0,0) -- (0.5,0); &
        \node[align=left] {\tiny Subdiffusive divergence (Theorem \ref{thm:main_thm_2})}; \\
      \draw[magenta!80!black, line width=1pt] (0,0) -- (0.5,0); &
        \node[align=left] {\tiny Diffusive}; \\
    };
  \end{scope}
\end{tikzpicture}
    \caption{Summary of our findings for general $\gamma\in (0,2)$. Compared to the Gaussian case, all intervals shift to the left by $\gamma/2$. The blue region indicates a transition to logarithmic fluctuations whenever $\alpha$ scales like $\log(T)$. Compare this to the Gaussian case, for which logarithmic fluctuations only occur for the case $\xi=2$. Note that we only show upper bounds on the variance in all of our results.}
    \label{fig:placeholder}
\end{figure}
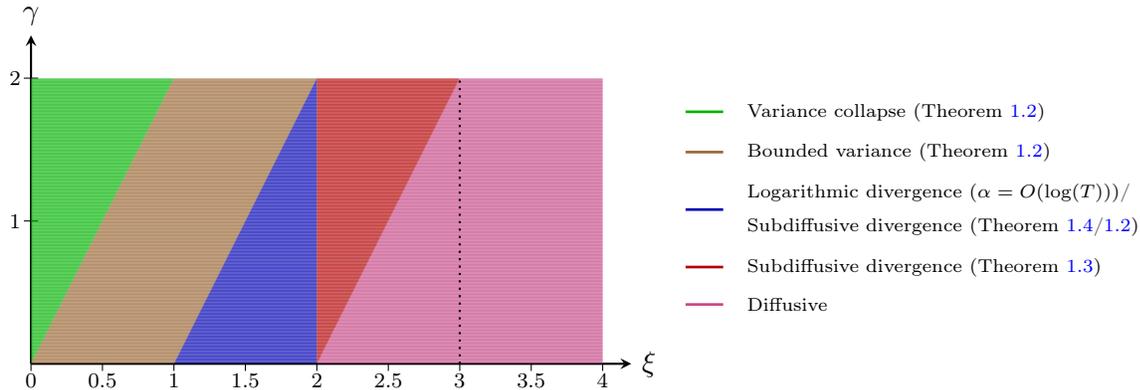

To prove Theorem \ref{thm:main_thm_1} we will argue by Gaussian domination via (GCI) that it suffices to study the case $f(x) = x^2$. In fact, one can also argue by an FKG-based extension of (GCI) recently introduced in \cite{BaMuSeVa23,Se24surface} that it suffices to prove Theorem \ref{thm:main_thm_2} and the following Theorems for $f(x) = |x|^{\gamma}$. Both reductions will be performed in Section \ref{sec:gci}. 

Let us continue by giving a brief intuition on why we expect the bounds in Theorem \ref{thm:main_thm_2} to be near-optimal.  A more detailed explanation of our proof strategy is given in Section \ref{sec:main_idea_of_the_proof}.
Suppose we can show that, at length scale $2^j$, the path fluctuates at most $2^{j(\xi-2)/\gamma}$. In other words, suppose it holds that
$$r_j = \sup\limits_{s,t \in [0,2^{j+1}]} |x_t -x_s| \le 2^{\frac{ j}{\gamma}(\xi-2)}.$$
Then, 
$$\int_0 ^{2^{j}}\de s \int_{2^{j}}^{2^{j+1}} \de t \frac{ | x_t - x_s|^{\gamma}}{1+|t-s|^{\xi}} \ge \int_0 ^{2^{j}}\de s \int_{2^{j}}^{2^{j+1}} \de t \frac{|x_t - x_s|^2}{2 r_j ^{2-\gamma} 2^{j\xi }}.$$
A quadratic form estimate (see Proposition \ref{prop:quadratic_form_domination}) then yields
$$\int_0 ^{2^{j}}\de s \int_{2^{j}}^{2^{j+1}} \de t \frac{ | x_t - x_s|^{\gamma}}{1+|t-s|^{\xi}} \ge \frac{2^{-j(\xi-2) }}{ 2 r_j^{2-\gamma}} \lt| \frac{1}{2^j} \int_0 ^{2^{j}} x_s \de s - \frac{1}{2^j} \int_{2^{j}} ^{2^{j+1}} x_s \de s \rt|^2.$$
The last estimate ought to imply that the  random variable
$$\frac{1}{2^j} \int_0 ^{2^{j}} x_s \de s - \frac{1}{2^j} \int_{2^{j}} ^{2^{j+1}} x_s \de s$$
has variance bounded by $r_j ^{2-\gamma} 2^{j(\xi-2)}$ under $\bbPh_{\alpha,T,\gamma,\xi}$. By assumption $r_j \le 2^{j(\xi-2)/\gamma}$, which yields that the variance is at most $2^{2j(\xi-2)/ \gamma}$. In particular, the averaged increments are bounded with very high probability by something strictly smaller than $2^{j(\xi-2)/\gamma}$. Since $x_t - x_s$ can be bounded pointwise by averaged increments (see Proposition \ref{prop:r_bound}), it is possible to take $j \to j+1$ and continue the argument recursively. Together with formula \eqref{equ:variance_decomposition}, which expresses $x_T$ as a sum of averaged increments, the previous discussion indicates that the process behaves sub-diffusively. 
\subsection{Connection to polaron models and related work}

Polaron models describe the interaction of a quantum particle with a quantized field. The 
basic intuition is that the coupling to the field increases the inertia of the particle, 
making it appear heavier. In principle, this effect could be so strong that the particle 
becomes localized, even though no external potential is applied. 
In 1986, Herbert Spohn published a paper \cite{Sp86} entitled `Roughening and 
pinning transition of the polaron'. In it, he argued that the (at the time conjectured) self-trapping of the Fröhlich polaron into a  
localized state does not happen at any coupling strength. He observed that the emergence of
a localized state  would correspond to a roughening transition 
in the functional integral representation. This would require long-range interactions in the 
`time'-argument of the paths. 
Since these interactions decay exponentially for the Fröhlich polaron, a roughening 
transition can not be expected from a statistical mechanics point of view. 

Since then, the impossibility of a localized ground state has been verified rigorously in 
several ways. To name just two, in \cite{BeSp05}, a calculation similar to one in 
\cite{Sp86} together with the proof of a functional central limit theorem showed that for 
path measures with sufficiently fast decay of interactions, the diffusively rescaled path 
measure has strictly positive diffusion constant. Together with the relations between 
diffusion constant and effective mass (see \cite{Sp87,DySp20}), this implies that no 
localization can occur. More recently, sharp, rigorous lower bounds have been proved on the 
effective mass of the Fröhlich polaron proper \cite{BrSe24}, also implying the absence of 
localization. 

A somewhat weaker effect is what is in \cite{Sp86} referred to as the pinning transition, 
and has since been much studied under the name of enhanced binding. It means that when 
placing the particle into an external potential that is too weak to lead to a localized state 
for the free particle, a sufficiently strong coupling to the field will result in the 
existence of a localized state. Enhanced binding is now well established in many different 
models, using both analytic \cite{HiSp01,HVV03,KoMa13} and functional integral \cite{BeSchSe24, HiSp01, HiSa08} methods. 

The topic of localization is much less understood. In order to explain what localization 
means, and how it connects with the mean square displacement of functional integrals, 
let us briefly introduce the basics of polaron models. 
The relevant Hamiltonian  is defined on the 
Hilbert space  $\caH = L^2(\bbR^d) \otimes \caF$, where $\caF$ denotes the Fock space. We 
denote by $\omega: \bbR^d \to [0,\infty)$  the energy-momentum relation of the field 
modes, and by $\de \Gamma(\omega)$ its differential second quantization. $a_k ^\dag$ and 
$a_k$ are the formal annihilation and creation operators satisfying the canonical 
commutation relations $[a^\dag _k , a_{k'}] = \delta(k-k')$. The energy of the coupled 
quantum system is given in suitable units by
$$H = - \frac{\Delta_x}{2} + \de \Gamma(\omega)  + \alpha\int \de k \frac{\hat{\rho}(k)}
{\sqrt{2 \omega(k)}} \lt( \e{ik \cdot x} a^\dag _k + \text{h.c.} \rt),$$
which is defined on a dense domain of $\caH$. The specific polaron model is determined by 
the parameters $\omega$ and $\rho$, the latter describing the coupling of the (smeared 
out) particle position to the field modes. 
Assuming $\hat{\rho} / \omega \in L^2 
(\bbR^d)$ and $\hat{\rho} / \sqrt{\omega} \in L^2 (\bbR^d)$, standard arguments can be 
used to show that $H$ is well-defined and bounded from below. 

The connection with the path integrals we study is given by the functional integral representation of the polaron. It is 
derived by applying a Feynman-Kac formula and then integrating over the field degrees of freedom \cite{HiLo20}. The relevant identity reads 
\begin{equation} \label{eq:FKF}
\lt\la f \otimes \Omega,\e{-TH} g \otimes \Omega \rt\ra = \int \de z f(z) \bbE^z \lt[ \exp \lt( \alpha \int_0 ^T \de t \int_0 ^T \de s W(x_t -x_s , |t-s|) \rt) g(x_T) \rt],
\end{equation}
where the inner product is in $\caH$, $\Omega$ denotes the Fock vacuum,  and where the 
pair interaction
\begin{equation} \label{eq:W}
W(x,t) = \int \de k \e{-|t| \omega(k)} \frac{|\hat{\rho}(k)|^2}{2 \omega(k)}\e{i k 
\cdot x}
\end{equation}
results from integrating out the field. Among the physically relevant choices of the 
parameters $\omega$ and $\rho$ we single out two: the choice of  $\omega(k) = 1, \hat 
\rho(k) = \frac{1}{|k|^2}$ (up to constants) in three dimensions gives the Fröhlich 
polaron, where $W(x,t) = \frac{1}{|x|} \e{-|t|}$. This is an example of a short range (in 
$t$) interaction, where localization can be ruled out. Another choice is (again up to 
constants and in three dimensions) $\omega(k) = |k|, \hat\rho(k) = 1$, which gives the
(singular) Nelson model. Here $W(x,t) = \frac{1}{|x|^2 + |t|^2}$, i.e.\ a long range pair 
potential. While in the precise form given, the Nelson model is too singular for the Hamiltonian 
to be well-defined, it can be regularized without impacting the decay of $W$ for large $t$. For 
details see \cite{Ne64,Fr73,Fr74,HiMa22,HaHiSi24}. 

The connection between the effect of the field on the movement of the particle 
can be described in terms of the mobility of paths in the functional integral. The best 
studied case, also here, is where the pair potential $W$ is short ranged, and then 
the relevant notion is the effective mass. 
This quantity, which has been the subject of much recent research (e.g. 
\cite{Se24,BrSe23,BrSe24,Br24,BeSchSe25,BaMuSeVa23,BePo23}), can be defined as follows: 
$H$ commutes with the total momentum operator $-i \nabla_x + \int k a_k ^\dag a_k \de k$, which means that it can be diagonalized in terms of the total momentum of the system. This yields
$$H = \int \de P H(P)$$
for a suitable family of operators $\{H(P): P \in \bbR^d\}$.
It can be shown that $E(P)= \inf \text{spec }H(P)$ is a real number for any combination of $\alpha,P$ and that $P \mapsto E(P)$ is radial. The effective mass is then given by
\[
m_\text{eff} ^{-1} = 2 \lim_{|P| \to 0} \frac{E(P)-E(0)}{|P|^2}.
\]
The connection with path integrals requires, in its rigorous form, 
first an infinite volume limit and then a central limit theorem. Namely, it can be shown that  
families of measures of the type 
\begin{equation}
\label{eq:P_T}
\bbP_T(\dd x) \propto \exp \lt( \alpha \int_{-T/2} ^{T/2} \de t \int_{-T/2} ^{T/2} \de s W(x_t -x_s , |t-s|) \rt) \bbP_0(\dd x)
\end{equation}
(where $\bbP_0$ is Brownian motion started at time $-T/2$)  converge locally weakly as $T \to \infty$ on the $\sigma$-algebra generated by increments; see \cite{BeSp05,MuVa19,BePo22} for various results of this 
type. The same references show that the limiting path measure converges to a Brownian motion under diffusive rescaling. In cases where $x \mapsto W(x,t)$ is rotationally invariant, the emergent diffusion matrix is a multiple of the unit matrix, and the relevant diffusion constant is the inverse $m_\text{eff} ^{-1}$ of the effective mass.
This was first argued in \cite{Sp86} and made fully rigorous in \cite{DySp20}.
Often, (for example in \cite{MuVa19}), the procedure of first taking the limit $T \to 
\infty$ and then the diffusive re-scaling is replaced by taking the diagonal limit 
\begin{equation}
    \label{equ:stochastic_effective_mass}
    \lim_{T \to \infty} \frac{1}{dT} \bbE_T[|x_{T/2} - x_{-T/2}|^2] =: \tilde m_{\text{eff}}^{-1},
\end{equation}
which also gives the connection to the mean square displacement. 
For short range potentials, this is expected to be equivalent, i.e.\
$m_{\text{eff}} = \tilde m_{\text{eff}} = \lim_{T \to \infty} \frac{s_T}{dT} $. 
This can be shown at 
least in cases like the Fröhlich polaron by a representation of the measure as a 
mixture of Gaussians \cite{MuVa19}. In more general cases, and in particular for the longer range 
potentials below, we are not aware of an argument that gives 
$m_{\text{eff}} = \tilde m_{\text{eff}}$. 

For longer range potentials we expect that $s_T$ grows slower than linearly for 
large $T$. Modulo the caveat of the previous paragraph about the use of the mean square 
displacement, this means that the effective mass of the related polaron model is infinite. 
It is tempting to conjecture that in these cases, the second derivative of the function 
$P \mapsto E(P)$ vanishes at $P=0$; but note that the theory of \cite{DySp20} connecting these 
two quantities explicitly requires a nonnegative diffusion constant. 

The third scenario is that $s_T$ does not diverge at all as $T \to \infty$, 
i.e.\ paths stay localized. Then $s_T$ stays bounded as $T \to \infty$. This case corresponds 
to localization of the polaron, although also here the mean square displacement is not quite
the right quantity to study. Instead, an external potential of the form $\eps |x|^2$ is added to
the polaron Hamiltonian $H$. In the functional integral, this means that we replace the measure 
$\bbP_{T}$ from \eqref{eq:P_T} by a measure 
$\bbP_{T, \eps}$ that has an additional factor of the form $\e{-\eps \int_{-T/2}^{T/2} |x_s|^2 \, \dd 
s}$ in its density with respect to $\bbP_0$. It also means that $H$ now has a ground state, and 
it is known (see e.g.\ \cite{Be03}) that an infinite volume limit of $\bbP_{T, \eps}$ 
exists 
as $T \to \infty$. By spectral theory, the Lebesgue density
of the distribution of $x_0$ converges to $\psi_\eps^2(z)$, 
the square of the ground state of $H$. Now letting $\eps \to 0$, the sequence $\psi_\eps$
of normalized vectors converges in the weak$^\ast$ topology along a subsequence. 
By definition \cite{Sp86}, localization occurs if there is a limit point that is not the zero vector. 
While it is plausible that this is equivalent with the mean square displacement $s_T$ staying 
bounded, we are not aware of a rigorous statement or proof of this correspondence. 

Let us now cover some of the related work both on the polaron and on its functional intergal representation. 
The most prominent example of self-interacting path measures is the Fröhlich polaron. One of the earliest studies that used path measure formalism is the celebrated result by Donsker and Varadhan \cite{DoVa83}, who explicitly calculated the asymptotic behaviour of the free energy as $\alpha \to \infty$ using large deviation techniques. Other works include \cite{BePo23,BaMuSeVa23,MuVa19,Se24,Sp87,Sp86}. Only recently has the Gaussian correlation inequality been applied to study the Fröhlich polaron \cite{Se24,BaMuSeVa23}. The developed machinery has been proven useful also in similar contexts, as can be seen in \cite{BeSchSe24,BeSchSe25}. In this work we will also rely on those techniques.

Considering general measures of the form $\bbP_{\mathbf{self}}$, it has already been identified in the literature that a time decay of
 $O(|t|^{-3})$ is a critical threshold. For example, in \cite{BeSp05} such time decay is needed to establish a CLT for a class of models. Moreover, in \cite{OsaSp99} it is shown that if the time-decay is slower than $O(|t|^{-3})$, then no unique infinite volume measure can exist. Thus, our results complement them in a sense that we show that there cannot exist a CLT with strictly positive diffusion constant, at least in the scenario of Theorem \ref{thm:main_thm_1}.
Other time regimes which have been extensively studied are mean field regimes, where large deviation methods become available. Here, interesting measures are also perturbed random walks, not just Brownian motion \cite{BoDeuSch93,Bo94,BrySla95,BoSch97}. In a similar vein, one can also replace Brownian motion $\bbP$ by a jump process supported on $\{ \pm 1\}$. In addition, the pair interaction $W(x_t - x_s,|t-s|)$ is replaced by $g(|t-s|) x_t x_s$.
This way one obtains the path integral formulation of the spin boson model, describing the interaction of a spin with a free quantized scalar field. Specializing to $g(t) = |t|^{-2}$, is has been proven recently that this model exhibits a phase transition in the coupling strength \cite{BeHiKrPo25}.

\subsection{Main idea of the proof and outline}
\label{sec:main_idea_of_the_proof}
Part of our method can be interpreted as a rigorous implementation of a Renormalization Group (RG). Hierarchical models were introduced in e.g.  \cite{Dy69,Ba72} to study phase transitions in long-range Ising models and calculate critical exponents.
Similar to \cite{BeHiKrPo25}, it can be shown in Dyson's hierarchical model \cite{Dy69} that a phase transition occurs in the (inverse) temperature for ferromagnetic spin interactions decaying like $t^{-\beta}$ for $1 < \beta <2$. 
For other models that can be studied rigorously using RG methods, we refer the interested reader to \cite{BauSlaGor19} and references therein.

We begin with a discussion of our approach for proving Theorem \ref{thm:main_thm_1}. There are two main ingredients: first, we require a path decomposition which lets us write $x_T$ in a deterministic fashion as a sum of averaged increments (cf. equation \eqref{equ:variance_decomposition} and Figure \ref{fig:dyadic_decomposition_tree}). This estimate is performed for dyadic $T$, i.e. we will assume $T = 2^t$ for some $t \in \bbN$. The decomposition of $x_T$ in terms of averaged increments is not particularly useful if one considers Brownian motion, as it is apriori not clear why the resulting random variables are helpful. If one considers $\bbPh_{\alpha,T,\xi}$, however, there is a quadratic form inequality first introduced in \cite{Se24} which becomes very handy (see Proposition \ref{prop:quadratic_form_domination}). In fact, it holds that $\bbPh_{\alpha,T,\xi}$ is stochastically dominated (in a suitable sense introduced in the next section) by another path measure $\bbP^\mathbf{hier} _{\alpha,T,\xi}$. This domination is the second ingredient in our proof; if one considers $\bbP^\mathbf{hier} _{\alpha,T,\xi}$, the averaged increments have a variance which can be calculated easily. In combination, the path decomposition and the change of measure suffice to prove Theorem \ref{thm:main_thm_1}.  

For our second result, Theorem \ref{thm:sub_diffusive_regime}, note that
$$|x_t - x_s|^{\gamma} = \frac{|x_t - x_s|^2}{|x_t - x_s|^{2-\gamma}}.$$
That is, one can interpret 
$$\frac{| x_t - x_s|^{\gamma}}{1+|t-s|^{\xi}}$$
as a Gaussian density, which is dampened by a path-dependent factor 
$$\lt( |x_t - x_s|^{2-\gamma}(1+|t-s|^{\xi}) \rt)^{-1}.$$
By assumption this additional factor decreases (recall the condition $\gamma <2$ in Theorem \ref{thm:main_thm_2}) in the size of the increments. Thus, if it would hold that
$$\sup\limits_{s,t \in [0,2^{j+1}]} |x_t - x_s| \le c_j,$$
then a simple quadratic form estimate yields
\begin{equation}
    \begin{split}
        \label{equ:bounded_fluctuation_domination}
    \int_0 ^{2^j} \de t \int_{2^j} ^{2^{j+1}} \de s \frac{| x_t - x_s|^{\gamma}}{1+|t-s|^{\xi}} &\geq \int_0 ^{2^j} \de t \int_{2^j} ^{2^{j+1}} \de s \frac{|x_t - x_s|^2}{2 c_j ^{2- \gamma }2^{j\xi}} \\
    &\succeq \frac{1}{2 c_j ^{2- \gamma }2^{j(\xi-2)}} \lt| \frac{1}{2^j} \int_0 ^{2^j} x_s \de s - \frac{1}{2^j}\int_{2^j} ^{2^{j+1}} x_t \de t \rt|^2.
    \end{split}
\end{equation}
This estimate is pointwise true whenever path increments are bounded by $c_j$.
Because Brownian increments are unbounded, the above bound only holds (for suitable $c_j$) with high probability. We rely on theory developed in \cite{Se24} to make our arguments rigorous: For each averaged increment, we split Brownian motion into a good part, where fluctuations are actually bounded, and a bad part. On the good part we can implement the previous heuristic rigorously and obtain a Gaussian measure that we can work with. By choosing appropriate $c_j$, we can also show that the bad part can be neglected, which suffices to show Theorem \ref{thm:sub_diffusive_regime}. Theorem \ref{thm:main_thm_2} is a rigorous implementation of a recursive Gaussian confinement  approach. In a first step, we restrict the increments of Brownian motion to fluctuate at most $\sqrt{T}$ away from the origin (up to $\log$ factors) over the entire time interval $[0,T]$. A simple Gaussian domination argument yields that $\bbPw_{\alpha,T,\gamma,\xi}$ is more strongly concentrated (in a way rigorously introduced in Section \ref{sec:gci}) than $\bbPw_{\alpha T^{\gamma/2 - 1},T,\xi}$. An application of Theorem \ref{thm:main_thm_1} yields that the increments under this measure are not only bounded by $T^{1/2}$ with high probability, but even (up to some corrections) by $T^{\beta/2}$ with high probability. Here, $\beta$ can be determined explicitly, depending on $\gamma$ and $\xi$. In case $\beta < 1$,
this implies that the increments of $\bbPh_{\alpha,T,\gamma,\xi}$ are even more strongly concentrated, which yields a new improved variance. Iterating this leads in the exponent to the fixed-point of the map
\[
\beta \to \xi -2 + \beta - \frac{\beta \gamma}{2},
\]
which evaluates to be $2(\xi-2)/\gamma$. For a more technical sketch of this approach we refer to the paragraph following Proposition \ref{prop:gaussian_measure_variance_bound}. Finally, we discuss the proof of Theorem \ref{thm:main_thm_3}. As before, we are trying to get rid of $|z|^\gamma$ and replace it by a square. The shortcoming of the proof of Theorem \ref{thm:sub_diffusive_regime} is that, at a higher level, we do not use that we expect already good concentration at lower levels. 
To understand what we mean by this, consider the base case $j=0$. In a first step we condition Brownian motion to have 'small' fluctuations on every interval of length $1$. This is true with probability at least $1-T^{-10}$ for the radius $ \sqrt{20\log(T)}$. The important thing here is to obtain a bound that tends to $1$ faster than $T^{-1}$, and this is why we need to choose $\alpha = O(\log(T)^{1/2})$. An application of Lemma \ref{lemma:lemma_31} lets us split Brownian motion into a good part, and a bad part.  On the bounded increment part of Brownian motion we find that the averaged increments of neighboring intervals can be dominated by a Gaussian with variance as in equation \eqref{equ:bounded_fluctuation_domination}. One can then bound increments of length $4$ by these averaged increments; we refer to Proposition \ref{prop:r_bound} for details.
These averages have very high probability to stay close to the origin, and so another application of \eqref{equ:bounded_fluctuation_domination} yields that $$\frac{1}{2} \int_0 ^2 x_s \de s - \frac 1 2 \int_2 ^4  x_t \de t$$
is bounded, with very high probability. Doing this recursively yields (again with high probability) a dominating Gaussian measure with density as defined in \eqref{equ:hierarchical_bm}. We perform the recursion up to level $t$ in case no failure arises (i.e. having all increments bounded), and stop whenever the inductive reasoning breaks down. In case the induction fails, we simply dominate our measure by a Brownian motion with twice the variance; this is possible by an application of (GCI). It is crucial that we \textbf{do not} follow the approach presented in \cite{Se24} and split the problem into smaller blocks. This approach will always lead to $O(T)$ failures, which excludes a bound that yields sub-diffusive behaviour. This is not a problem in case of the Fröhlich polaron (as it has exponential time decay in its pair potential), but is highly suboptimal if one considers algebraic decay in time.

We finish this Section by giving a brief outline for the remainder of this paper. In Section \ref{sec:gci} we introduce all necessary notation and the main tools we need to obtain our results. Finally, we will prove Theorem \ref{thm:main_thm_1}. In the following Section we prove Theorem \ref{thm:sub_diffusive_regime}. In the following two sections the more technically involved proofs of Theorem \ref{thm:main_thm_2} and Theorem \ref{thm:main_thm_3} are given, respectively.
\section{Hierarchical path measures}
\label{sec:hierarchical_path_measures1}
In all what follows we will take $T= 2^t$ for some $t \in \bbN$. In analogy to classical spin systems we write
$$\sigma_i := \int_i ^{i+1} x_s \de s$$
and define for $l \in 2 \bbN$, $j \in \bbN$
$$\sigma_j ^{l} := \sum\limits_{k= j} ^{j+l-1} \sigma_k = \int_{j} ^{j+l} x_s \de s,  \: \: \: \: \: \: \: \: \: \: \: \: \: \: \:  s_j ^l := \sigma_{j+{l/2}} ^{l/2} - \sigma_{j} ^{l/2} = \int_{j+l/2} ^{j+l} \!\!\!\!\!\! x_s \de s - \int_{j} ^{j+l/2} \!\!\!\!\!\!\!\!\! x_s \de s , \:\:\:\:\:\:\:\:\:\:\:\:\:\:\: \overline{s}_j ^l := \frac{1}{l} s_j ^l.$$
In case $l=1$ we additionally define $\overline{s}_j ^1 := s_j ^1 := \sigma_j$.
With the introduced notation at hand we can find a representation of $x_T$ in terms of $\{\overline{s}_v ^{2^l}:v \in \{0,T-2^l\}, l \le t\}$; see also Figure \ref{fig:dyadic_decomposition_tree}. Indeed, we claim that
\begin{equation}
    \label{equ:variance_decomposition}
    x_T = x_T - \sigma_{T-1} + \sum\limits_{j = 1} ^{t} \overline{s}_0 ^{2^j} +\sum\limits_{l = 1} ^{t-1} \overline{s}_{T- 2^l} ^{2^l} + \sigma_0.
\end{equation}
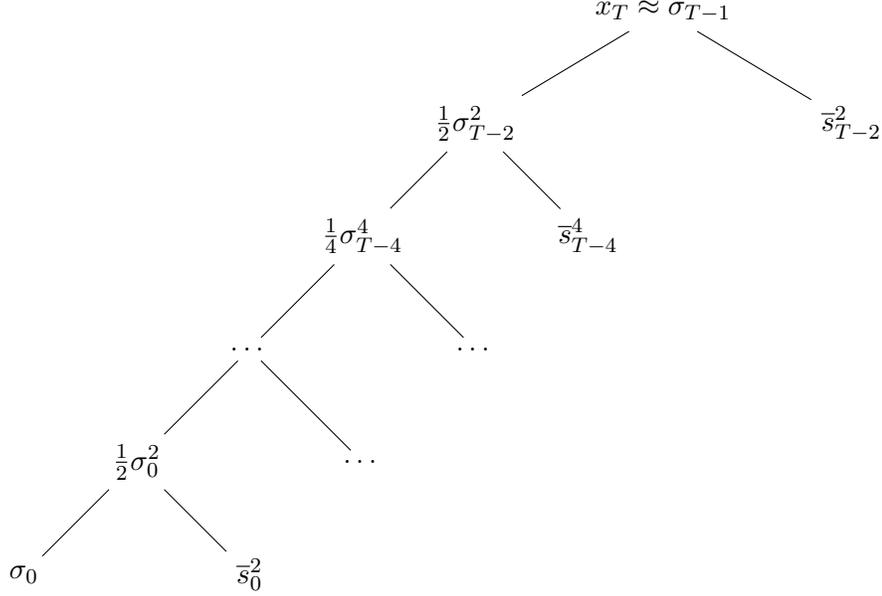
\begin{figure}
    \centering
    \begin{tikzpicture}[
  level 1/.style={sibling distance=50mm},
  level 2/.style={sibling distance=30mm},
  every node/.style={align=center}
]
\node {$x_T \approx \sigma_{T-1}$}
  child {node {$\frac{1}{2}\sigma_{T-2} ^2$}
    child {
  node {$\frac{1}{4}\sigma_{T-4}^4$}
  child {
  node {\dots}
  child {
  node {$\frac{1}{2}\sigma_0 ^2$}
  child {node {$\sigma_0$}}
  child {node {$\overline{s}_0 ^2$}}
}
  child {node {\dots}}
}
  child {node {\dots}}
}
    child {node {$\overline{s}_{T-4}^4$}}
  }
  child {node {$\overline{s}_{T-2}^2$}
  };
\end{tikzpicture}

    \caption{Dyadic decomposition of the path. The sum of a node's children equals the node's value. At each level, we keep all $\overline{s}$ terms, since those can be estimated directly. This yields a decomposition of $\sigma_{T-1}$ into $\log(T)$ terms.}
    \label{fig:dyadic_decomposition_tree}
\end{figure}

To see that this equality holds, we rely on the fact that, for all $j \in \bbN$,
$$2^{-j} - 2^{-(j+1)} = 2^{-(j+1)}.$$
In a first step we write
$x_T = x_T - \sigma_{T-1} + \sigma_{T-1}$. By definition it then holds that
$$\sigma_{T-1} = \sigma_{T-1} - \frac 1 2 \sigma_{T-2} ^2 + \frac 1 2 \sigma_{T-2} ^2 = \overline{s}_{T-2}^2 + \frac 1 2 \sigma_{T-2} ^2.$$
We can continue this recursively; i.e., in general we have with $n = 2^l$ for arbitrary $l \in \bbN$
$$\frac{1}{n} \sigma_{T- n}^{n} = \frac{1}{n} \sigma_{T- n}^{n} - \frac{1}{2n} \sigma_{T- 2n}^{2n} + \frac{1}{2n} \sigma_{T- 2n}^{2n} = \frac{1}{2n} s^{2n}_{T-2n} + \frac{1}{2n} \sigma^{2n}_{T-2n} = \overline
s^{2n}_{T-2n} + \frac{1}{2n} \sigma^{2n}_{T-2n}.$$
Once we arrive at $\sigma^{T}_{0}$ (i.e. once $n=2^{t-1}$) we can perform the same procedure backwards. In formulas we find
$$\frac 1 T \sigma_0 ^T = \frac 1 T \sigma_0 ^T - \frac 2 T \sigma_0 ^{T/2} + \frac 2 T \sigma_0 ^{T/2} = \frac 1 T \lt( \sigma_{T/2} ^T - \sigma_0 ^{T/2} \rt) + \frac 2 T \sigma_0 ^{T/2} = \overline{s}_0 ^T + \frac 2 T \sigma_0 ^{T/2}.$$
Doing this until one arrives at $\overline{s}_0 ^2$ and collecting all terms justifies \eqref{equ:variance_decomposition}. See also Figure \ref{fig:dyadic_decomposition_tree}. We will use this decomposition in Lemma \ref{lemma:variance_decomposition}; the approach there differs in the fact that we are interested in the mean-square displacement. This means that we have to take care of the covariances between two different averaged increments. For this it suffices to use Cauchy-Schwarz.

As alluded to earlier, the random variables $\overline{s}_0 ^{2^l}$ are complicated and seem unsuited to do explicit calculations with. In the next Section we introduce a form of stochastic domination and show that $\bbPh_{\alpha,T,\xi}$ is dominated in a suitable sense by a measure where the averaged increments $\overline{s}_0 ^{2^l}$ have a variance that can be calculated easily. In fact, the dominating measure is given by
\begin{equation}
    \label{equ:hierarchical_bm}
    \begin{split}
        \bbP^{\mathbf{hier}}_{\alpha,T,\xi}(\de x) &\propto \exp \lt( - 
 \alpha \sum\limits_{l=1} ^{t} \frac{1}{2^{\xi l+1}} |  s_{0} ^{2^l} |^2 - \alpha \sum\limits_{l=1} ^{t-1} \frac{1}{2^{\xi l+1}} |  s_{T - 2^l} ^{2^l} |^2 \rt) \bbP (\de x) \\
 &\propto \exp \lt( - 
 \alpha \sum\limits_{l=1} ^{t} \frac{1}{2^{(\xi-2)l +1}} |  \overline{s}_{0} ^{2^l} |^2 - \alpha \sum\limits_{l=1} ^{t-1} \frac{1}{ 2^{(\xi-2)l +1}} |  \overline{s}_{T - 2^l} ^{2^l} |^2 \rt) \bbP (\de x).
    \end{split}
\end{equation}
It is clear that, for $v \in \{0,T-2^l\}$ and $l \le t$,
$$\bbE^{\bbP^{\mathbf{hier}}_{\alpha,T,\xi}}\lt[ \lt( \overline{s}_v ^{2^l} \rt)^2 \rt] \le 2^{(\xi-2)l}.$$

\subsection{(GCI), FKG-GCI and quadratic form domination}
\label{sec:gci}
The main Theorem we use for dominating one measure by another is a functional version of the Gaussian correlation inequality, first proven in \cite{Ro14}. We recall that a function $f$ is called quasi-concave iff $\{f \ge c\}$ is a convex set for all $c > 0$. For two probability measures $\mu,\nu$ on $C([0,T];\bbR^d)$ we write
$\mu \preceq \nu$
whenever $\mu(A) \ge \nu(A)$ for all sets $A \subseteq C([0,T];\bbR^d)$ symmetric and convex.
\begin{theorem}[GCI]
\label{thm:GCI}
    Let $f-g$ be a symmetric and quasi-concave function (or the uniformly bounded limit of such). 
    If 
    $$\mu^{(g)}(\de x) \propto \e{g(x)}\mu (\de x)$$
    is a centred Gaussian measure, then
    $$\mu^{(f)} \preceq \mu^{(g)}.$$
\end{theorem}
\begin{proof}
    See e.g. \cite{BeSchSe25}.
\end{proof}
It will occur that $f-g$ is only symmetric and quasi-concave on a certain subset of the space we integrate over. We now introduce a useful tool to restrict Gaussian measures to a compact subset in a way which does not violate (GCI). The numerical value for the enlargement of the set $K$ can be calculated explicitly \cite{Sch25}.
\begin{lemma}[{\cite[Lemma 3.1]{Se24}}]
    \label{lemma:lemma_31}
   Let $\mu$ be a centred Gaussian measure on a finite-dimensional real vector-space $M$ and let $K \subseteq M$ be a symmetric convex set with $\mu(K) > 1- \delta$ for some $\delta \leq 0.1$. Then, there exists a decomposition 
$$\mu = (1-\delta') \nu + \delta' \overline{\nu}$$
of $\mu$ into a mixture of probability measures such that:
\begin{enumerate}
    \item $\delta' \leq \delta$,
    \item $\supp(\nu)\subseteq 64 K$,
    \item $\nu \preceq \mu$,
    \item $\overline{\nu} \preceq \mu^{\times 2}$.
\end{enumerate}
\end{lemma}

In the following we will show a quadratic form bound which is essential for the remainder of this paper. In essence, it allows us to use (GCI) to switch from measures like $\bbPh_{\alpha,T,\xi}$ to $\bbP^\mathbf{hier}_{\alpha,T,\xi}$. The main advantage of the second measure is that it can be used in combination with Lemma \ref{lemma:variance_decomposition} to obtain explicit bounds on the variance of the path at time $T$.

Fix two intervals $I_1 =[i,j]$ and $ I_2 = [k,l]$. Abbreviate
    $$Q_{I_1,I_2}(x) := \int\limits_{I_1}  \int\limits_{I_2}  \lt| x_t - x_s \rt|^2 \de t \de s,$$
    $$\Tilde{Q}_{I_1,I_2} := \lt| x_{I_1} - x_{I_2}\rt| ^2.$$
Let $Q,P$ be two quadratic forms on $X$. We write
$Q \succeq P$ in the sense of quadratic forms if $Q - P$ is a (non-negative) quadratic form, as well. 
\begin{proposition}
    \label{prop:quadratic_form_domination}
    \label{prop_quadratic_forms}
    In the sense of quadratic forms, 
    $$Q_{I_1,I_2} \succeq \Tilde{Q}_{I_1,I_2}.$$
\end{proposition}
\begin{proof}
    In this proof only we will use the notation
    $$x_{I_1} = \sigma_i ^{|j-i|} = \int_{I_1}  x_s \de s$$
    as it makes the notation easier to read. The claim is shown if the equation
    \begin{equation}
        \nonumber
        \begin{split}
             Q_{I_1, I_2}(x) -  \Tilde{Q}_{I_1,I_2}(x) &=  \int\limits_{I_1}  \lt| x_t - x_{I_1} \rt|^2 \de t +  \int\limits_{I_2}  \lt| x_t - x_{I_2} \rt|^2 \de t
        \end{split}
    \end{equation}
    is verified. By definition
\begin{equation}
    \nonumber
    \begin{split}
         Q_{I_1, I_2}(x) -  \Tilde{Q}_{I_1,I_2}(x) =& \int_{I_1} \int_{I_2}  \lt| x_t - x_s  \rt|^2 \de t \de s - \lt| x_{I_1}- x_{I_2} \rt|^2 \\
         &= \int_{I_1} \int_{I_2} \langle x_t - x_s , x_t - x_s \rangle \de t \de s - \int_{I_1} \int_{I_2} \langle x_s - x_t, x_{I_1}- x_{I_2} \rangle \de t \de s \\
         &= \int_{I_1} \int_{I_2} \langle x_t - x_s, x_t - x_s + x_{I_1} - x_{I_2}  \rangle \de t\de s \\
         &=  \int_{I_2} \langle x_t, x_t - x_{I_2}  \rangle \de t + \int_{I_1}  \langle x_s,  x_s - x_{I_1} \rangle \de s.
    \end{split}
\end{equation}
Noticing that
$$\int_{I_1} \langle x_{I_1}, x_{I_1} - x_s \rangle \de s = 0$$
and similarly for $I_2$, it follows that
\begin{equation}
    \nonumber
    \begin{split}
       Q_{I_1, I_2}(x) -  \Tilde{Q}_{I_1,I_2}(x) &= \int_{I_2} \langle x_t - x_{I_2}, x_t - x_{I_2}  \rangle \de t \de s + \int_{I_1}  \langle x_s - x_{I_1}, x_s - x_{I_1} \rangle \de t \de s \\
       &= \int_{I_2} \lt| x_t - x_{I_2} \rt|^2 \de t + \int_{I_1} \lt| x_s - x_{I_1} \rt|^2 \de s,
    \end{split}
\end{equation}
which shows the claim.
\end{proof}

\begin{lemma}
    \label{lemma:simple_density_replacement}
    There exists $c>0$ such that, for every $\alpha>0$ with $\beta:= c \alpha$ it holds that
    \begin{equation}
        \label{equ:simple_domination}
        \bbPh_{\alpha,T,\xi} \preceq \bbP^{\mathbf{hier}}_{\beta,T,\xi}.
    \end{equation}
    In particular, 
    \begin{equation}
        \label{equ:simple_variance_inequality}
        \bbE^{\bbPh_{\alpha,T,\xi}} \lt[ | x_T |^2 \rt] \le \bbE^{\bbP^{\mathbf{hier}}_{\beta,T,\xi}} \lt[ | x_T |^2 \rt].
    \end{equation}
\end{lemma}

\begin{proof}
    By assumption there exists $\zeta$ so that
    $$z \mapsto  -\alpha f(z) + \alpha \zeta z^2$$
    is a symmetric and quasi-concave function. Set $\beta := \alpha \zeta$. By approximation via Riemann sums it holds that
    $$-\int_0 ^T \de t \int_0 ^T \de s \frac{f(x_t - x_s)}{1+ |t-s|^\xi } + \frac{|x_t - x_s|^2}{1+ |t-s|^\xi }$$
    is a (universally bounded limit of) product of symmetric and quasi-concave functions. Moreover, 
    $$\int_0 ^T \de t \int_0 ^T \de s \frac{|x_t - x_s|^2}{1+ |t-s|^\xi } \succeq  \sum\limits_{l=1} ^{t} \frac{1}{2^{\xi(l+1)}} |  s_{0} ^{2^l} |^2 + \sum\limits_{l=1} ^{t-1} \frac{1}{2^{\xi(l+1)}} |  s_{T - 2^l} ^{2^l} |^2$$
    holds in the sense of quadratic forms thanks to Proposition \ref{prop:quadratic_form_domination} and noting that $1+|t-s|^\xi \ge 2^{\xi(l+1)}$ whenever $l \ge 1$ and $t,s$ being in the same block of size $2^l$\footnote{By a \textit{block} of size $2^l$ we mean any interval $[n 2^l,(n+1)2^l]$ for $n \in \bbN$.}. This shows the inequality in \eqref{equ:simple_domination}. Equation \eqref{equ:simple_variance_inequality} is a direct consequence of (GCI). 
\end{proof}

Finally, we introduce an extension of (GCI), called the FKG-GCI inequality. This inequality was used in \cite{Se24surface} and \cite{BaMuSeVa23} to study the Fröhlich polaron as well as random surface models. The FKG-GCI inequality extends (GCI) to mixtures of Gaussian measures. This is possible whenever the mixing measure satisfies a log-supermodularity property. 
Take $\eta \in (0,2)$. Then, there exists a probability measure $\mu_\text{mix}$ so that 
\[
\e{-x^\eta} = \int \mu_\text{mix}(\de s) \e{- x^2 / 2s}
\]
holds \cite{We87}.
By Fourier inversion it is seen that $\mu_\text{mix}$ is absolutely continuous with respect to Lebesgue measure. Denote the corresponding density by $f_\text{mix}$. It is then possible to write for any finite index set $\caA \subseteq \{ (s,t) \in \bbR^2 _+: s < t\}$ and time decay function $g \ge 0$

\begin{equation}
    \nn
    \begin{split}
        &\exp \lt( - \sum\limits_{(s,t) \in \caA} |x_{t} - x_s|^{\eta} g(|t-s|) \rt) \\
        &= \int  \prod\limits_{(s,t) \in \caA } \de z_{(s,t)}f_\text{mix}(z_{(s,t)}) \exp \lt(- \frac{| x_t - x_s|^2 g(|t-s)|)^{2/\eta}}{2 z_{(s,t)}} \rt),
    \end{split}
\end{equation}
which implies that the measure 
$$\mu (\de x) \propto \exp \lt( - \sum\limits_{(s,t) \in \caA} |x_{t} - x_s|^{\eta} g(|t-s|) \rt) \bbP(\de x)$$
can be written as
\begin{equation}
    \label{equ:finite_dim_gaussian_mixture}
    \mu = \int \Gamma(\de \xi) \gamma_\xi.
\end{equation}
In the above, $\Gamma$ is a probability measure on $\bbR^{|\caA|}$ with log-supermodular Lebesgue density, and $\gamma_\xi$ are Gaussian probability measures. Indeed for products of one-dimensional densities, log-supermodularity always holds (though one needs to rescale the component densities to be probability measures so it is not entirely obvious; see \cite[Theorem 2.10(a)]{Se24surface}). 

\begin{proposition}
    $$\bbE^{ \bbPw_{\alpha,T,f,\xi} } [ |x_T| ^2] \le \bbE^{ \bbPh_{\zeta\alpha,T,\gamma,\xi} } [ |x_T| ^2].$$
\end{proposition}

\begin{proof}
    The reweighting of Brownian motion by 
    $$\exp \lt(- \int_0 ^T \de t \int_0 ^T  \de s \frac{f(x_t - x_s)}{1+|t-s|^{2+\gamma/2}} \rt)$$
    can be written as a reweighting by 
    $$\exp \lt( \int_0 ^T \de t \int_0 ^T  \de s \frac{-f(x_t - x_s) + \zeta |x_t - x_s|^{\gamma}}{1+|t-s|^{\xi}} \rt)$$
    with respect to the measure $\bbPh_{\zeta\alpha,T,\gamma,\xi}$. Approximating the reweighting of $\bbPh_{\zeta\alpha,T,\gamma,\xi}$ by Riemann sums, it holds by equation \eqref{equ:finite_dim_gaussian_mixture} that the approximation of $\bbPh_{\zeta\alpha,T,\gamma,\xi}$ can be written as (log-supermodular) mixtures of centered Gaussian densities. Let $\caA$ be the index set of the Riemann sum approximation. Because the density
    $$\exp \lt( \sum\limits_{(s,t) \in \caA}  \frac{\zeta|x_t - x_s|^{\gamma} -f(x_t - x_s)}{1+ |t-s|^{\xi}} \rt)$$
    is symmetric and quasi-concave on path space by assumption, the claim follows by the FKG-GCI inequality \cite[Theorem 2.7]{Se24surface}.
\end{proof}

\subsection{Variance decomposition}
\label{sec:proof_of_thm12}
This section is devoted to proving Theorem \ref{thm:main_thm_1}. We start by presenting a simpler argument which yields suboptimal bounds.
It is easily seen by (GCI) that
    $$\bbPh_{\alpha,T,\xi} \preceq \bbP^* _{\alpha,T}$$
    with
    $$\bbP^* _{\alpha,T}(\de x) \propto \exp \lt( \frac{\alpha}{T^\xi} \int_0  ^T \de t \int_0 ^T  \de s | x_t - x_s |^2 \rt) \bbP(\de x).$$
    By a simple re-scaling argument and results from \cite{Se24} we find that 
    $$\bbE^{\bbP^* _{\alpha,T}} [|x_T|^2] /T = \bbE^{\bbP^* _{T^{3-\xi}\alpha,1}} [|x_1|^2] = O_{T \to \infty}(T^{-(3-\xi)/2}).$$
We conclude that
$$\bbE^{\bbP^* _{\alpha,T}} [|x_T|^2] \le O (T^{(\xi-1)/2}).$$
This already suffices to show the second claim in Theorem \ref{thm:main_thm_1}. Observe, however, that this estimate does not yield localized behaviour for any $\xi > 1$. Compare this to the prediction made in \eqref{equ:gaussian_calculation_pred}, which predicts localized behaviour for $\xi < 2$.
The problem of the one-step estimate is that locally, the time-decay is much smaller than $T^\xi$ under $\bbPh_{\alpha,T,\xi}$. It is therefore
natural to expect that the path decomposition obtained in \eqref{equ:variance_decomposition} together with tighter time-decay estimates for each averaged increment improves the order with which the normalized mean-square displacement vanishes.
Let us now make this multi-scale approach rigorous. For the remainder of this section 
we treat slightly more general measures compared to $\bbP^{\mathbf{hier}}_{\alpha,T,\xi}$; this generalization will become useful later on. Let
$$\varphi_{\alpha,T} (\de x) \propto \exp \lt(  - 
 \alpha \sum\limits_{l=1} ^{t} \frac{1}{2 A_l ^2} |  \overline{s}_{0} ^{2^l} |^2 - \alpha \sum\limits_{l=1} ^{t-1} \frac{1}{2 A_l ^2} |  \overline{s}_{T - 2^l} ^{2^l} |^2 \rt) \bbP (\de x)$$
 for some positive constants $A_1,\dots,A_t$.  
 At the beginning of this section, we will work with general $A_j$ and only specialize in the proof of Theorem \ref{thm:main_thm_1} to $A_j = 2^{(\xi-2)j/2}$. At the beginning of Section \ref{sec:hierarchical_decomposition}, we finally define specific $A_j$ which will remain unchanged for the remainder of this work. 
 
\begin{lemma}
    \label{lemma:variance_decomposition}
    Take $A^* := \max\{A_j ^2: j \le t\}$. Then,
    \begin{equation}
    \label{equ:full_model_bound}
    \begin{split}
        \bbE^{\varphi_{\alpha,T}} \lt[  |\sigma_{T-1} - \sigma_0|^2 \rt] &\le \sum\limits_{i=1} ^{t-1} \bbE^{\varphi_{\alpha,T}} \lt[ \big| \overline{s}_{T-2^i} ^{2^i}  \big|^2 \rt] + \sum\limits_{i=1} ^t \bbE^{\varphi_{\alpha,T}} \lt[ \big|  \overline{s}_{0} ^{2^i}  \big|^2 \rt] \\
        &+  2 \sum\limits_{i=1} ^t \sum\limits_{j=1} ^{t-1} \bbE^{\varphi_{\alpha,T}} \lt[ |\overline{s}_{T-2^j} ^{2^j}|^2 \rt]^{1/2}  \bbE^{\varphi_{\alpha,T}} \lt[|\overline{s}_{0} ^{2^i}|^2 \rt]^{1/2}  \\
        &\le 2 \alpha^{-1} \sum\limits_{i=1} ^t A_i ^2 + 4 \alpha^{-1} A^*t^2.
    \end{split}
\end{equation}
\end{lemma}
\begin{proof}
   Recall equation \eqref{equ:variance_decomposition}; this directly implies that we have the equality
   \begin{equation}
       \label{equ:variance_decomp_2}
       \begin{split}
           \bbE^{\varphi_{\alpha,T}} \lt[  |\sigma_{T-1} - \sigma_0|^2 \rt] &=  \bbE^{\varphi_{\alpha,T}} \Big[ \Big| \sum\limits_{j=1} ^{t-1}  \overline{s}_{T-2^j} ^{2^j}  \Big|^2 \Big] \\
           &+  \bbE^{\varphi_{\alpha,T}} \Big[ \Big| \sum\limits_{i=1} ^t \overline{s}_{0} ^{2^i}  \Big|^2 \Big] + 2 \sum\limits_{i=1} ^t \sum\limits_{j=1} ^{t-1} \bbE^{\varphi_{\alpha,T}} \lt[ \overline{s}_{T-2^j} ^{2^j} \overline{s}_{0} ^{2^i} \rt].
       \end{split}
   \end{equation}
   The cross terms can be estimated by Cauchy-Schwarz; indeed,
   $$\bbE^{\varphi_{\alpha,T}} \lt[ \overline{s}_{T-2^j} ^{2^j} \overline{s}_{0} ^{2^i} \rt] \le \bbE^{\varphi_{\alpha,T}} \lt[ |\overline{s}_{T-2^j} ^{2^j}|^2 \rt]^{\frac 1 2}  \bbE^{\varphi_{\alpha,T}}\lt[ |\overline{s}_{0} ^{2^i}|^2 \rt]^{\frac 1 2}  \le \alpha^{-1}A^*$$
   as $\varphi_{\alpha,T}$ specifies the variance of all averaged increments. As there are $(t-1)t \le t^2$ cross-terms, we find
   $$2 \sum\limits_{i=1} ^t \sum\limits_{j=1} ^{t-1} \bbE^{\varphi_{\alpha,T}} \lt[ \overline{s}_{T-2^j} ^{2^j} \overline{s}_{0} ^{2^i} \rt] \le \alpha^{-1}2t^2 A^*.$$
   The first two sums in \eqref{equ:variance_decomp_2} can be treated similarly, leading to the bound
   $$\bbE^{\varphi_{\alpha,T}} \lt[  |\sigma_{T-1} - \sigma_0|^2 \rt] \le \sum\limits_{i=1} ^{t-1} \bbE^{\varphi_{\alpha,T}} \lt[ \big|  \overline{s}_{T-2^i} ^{2^i}  \big|^2 \rt] + \sum\limits_{i=1} ^t \bbE^{\varphi_{\alpha,T}} \lt[ \big|  \overline{s}_{0} ^{2^i}  \big|^2 \rt] + 2\alpha^{-1}(t+t^2) A^*.$$
   The second inequality in \eqref{equ:full_model_bound} follows again by using the variance given in the Gaussian measure $\varphi_{\alpha,T}$.
\end{proof}
\begin{proof}[Proof of Theorem \ref{thm:main_thm_1}]
    Let us first discuss the case $\xi \ge 2$. Here, the proof is a direct application of Lemma \ref{lemma:variance_decomposition}.
    By Lemma \ref{lemma:simple_density_replacement} there exists $\beta >0$ so that
    $$\bbE^{\bbPh_{\alpha,T,\xi}}[|\sigma_{T-1} - \sigma_0|^2] \le \bbE^{\bbP^{\mathbf{hier}} _{\beta,T,\xi}}[|\sigma_{T-1} - \sigma_0|^2].$$
    Moreover, $\bbP^{\mathbf{hier}}_{\beta,t,\xi}$ simply is $\varphi_{\beta,T}$ with $A_j ^2 = 2^{(\xi -2)j}$. It then follows from equation \eqref{equ:full_model_bound} and results from \cite[Section 4]{Se24} that
    \begin{equation}
        \nn
        \begin{split}
            \frac 1 4 \bbE^{\bbPh_{\alpha,T,\xi}}[|x_T|^2] &\le \bbE^{\bbPh_{\alpha,T,\xi}}[ |x_T - \sigma_{T_1}|^2 ] + \bbE^{\bbP^\mathbf{hier} _{\beta,T,\xi}}[ |\sigma_0|^2 ] + \bbE^{\bbPh_{\alpha,T,\xi}} \lt[  |\sigma_{T-1} - \sigma_0|^2 \rt]
            \\
            &\le 3 \min (1,\alpha^{- 1/2})+  \frac 2 \beta \sum\limits_{i= 1} ^t A_i ^2  + 4\beta^{-1} A^* \log(T)^2 \\
            &= 3\min (1,\alpha^{- 1/2}) + \frac 2 \beta  \sum\limits_{l= 1} ^t 2^{(\xi -2)l }  + 4 \beta^{-1} T^{\xi-2} \log(T)^2.
        \end{split}
    \end{equation}
    As $l \le t$ the claim follows.
    In case of $\xi<2$, we find that the $A_j$ are summable (as they decay exponentially). Together with the first bound obtained in \eqref{equ:full_model_bound}, the result follows by summing all $A_j$ and the same argument as before.
\end{proof}
\begin{remark}
    \label{rem:simpler_proof} 
    Studying the proof of Theorem \ref{thm:main_thm_1}, we find that the bound on the mean-square displacement improves from $T^{(\xi-1)/2}$ to $T^{\xi-2} \log(T)^2$. In particular, we obtain sub-diffusivity for $\xi \in (0,3)$ and  localization for $\xi \in (1,2)$. For the case $\xi=2$, our bound produces the estimate $\log(T)^2$. This is slightly weaker compared to the predicted bound of $\log(T)$ \cite{Sp86}.
\end{remark}

\section{Proof of Theorem \ref{thm:sub_diffusive_regime}}
This Section is devoted to proving Theorem \ref{thm:sub_diffusive_regime}. Let us continue to write $A_j = 2^{j(\xi-2)/2}$. Our approach for the proof will rely on the previously obtained decomposition
\begin{equation}
    \label{equ:same_measure_full_model_bound}
    \begin{split}
        \bbE^{\bbPw_{\alpha,T,\gamma,\xi}} \lt[  |\sigma_{T-1} - \sigma_0|^2 \rt] &\le  \bbE^{\bbPw_{\alpha,T,\gamma,\xi}} \lt[ \big| \sum\limits_{i=1} ^{t-1}\overline{s}_{T-2^i} ^{2^i}  \big|^2 \rt] +  \bbE^{\bbPw_{\alpha,T,\gamma,\xi}} \lt[ \big| \sum\limits_{i=1} ^t  \overline{s}_{0} ^{2^i}  \big|^2 \rt] \\
        &+  2 \sum\limits_{i=1} ^t \sum\limits_{j=1} ^{t-1} \bbE^{\bbPw_{\alpha,T,\gamma,\xi}} \lt[ |\overline{s}_{T-2^j} ^{2^j}|^2 \rt]^{1/2}  \bbE^{\bbPw_{\alpha,T,\gamma,\xi}} \lt[|\overline{s}_{0} ^{2^i}|^2 \rt]^{1/2}  \\
        &\le 2 \alpha^{-1} \sum\limits_{i=1} ^t A_i ^2 + 4 \alpha^{-1} A^* t^2 \\
        &\le \alpha^{-1}C t^2 A^*
    \end{split}
\end{equation}
for some $C>0$.
The first two inequalities in \eqref{equ:same_measure_full_model_bound} follow simply by Cauchy-Schwarz. Note that the first estimate is better in case one can show that the series all converge. Indeed, in this case one can avoid the logarithmic divergence present in the final bound. We will use the second bound in the regime $\xi > 1+\gamma/2$ and the first bound in case $\xi < 1+ \gamma/2$. Moreover, note that it suffices to bound  $\bbE^{\bbPw_{\alpha,T,\gamma,\xi}} \lt[  |\sigma_{T-1} - \sigma_0|^2 \rt]$ by the same argument as before. 

In both regimes of $\xi$ we need to bound $A_j$ for $\gamma$ fixed. Finally, recall that the we can study the measure $\bbPw_{\alpha,T,\gamma,\xi}$ instead of $\bbPh_{\alpha,T,f,\xi}$ thanks to the results in Section \ref{sec:gci}.
Define for $i \le t$ the sets 
$$B_i = \lt\{ \sup\limits_{0 \le s < t \le 2^i} |x_t-x_s| \le \sqrt{2^i 20 \log(2^i)}\rt\}.$$
\begin{proposition}
    For any $i \le t$,
    $\bbP(B_i) \ge 1- 2^{-10i}$.
\end{proposition}
\begin{proof}
    By Brownian rescaling it suffices to show
    $$\bbP \lt(\sup\limits_{0 \le s < t \le 1} |x_t-x_s| \ge \sqrt{20 \log(2^i)} \rt).$$
    Moreover, 
    \begin{equation}
        \nn
        \begin{split}
            \bbP \lt(\sup\limits_{0 \le s < t \le 1} |x_t-x_s| \ge \sqrt{20 \log(2^i)} \rt) &\le \bbP \lt(\sup\limits_{0 < t \le 1} |x_t| \ge 2\sqrt{20 \log(2^i)} \rt) \\
            &= 2\bbP \lt(x_1 \ge 2\sqrt{20 \log(2^i)} \rt) \\
            &\le 2^{-10i}.
        \end{split}
    \end{equation}  
    Here, the last equation follows by the symmetry of the distribution of $x_1$ and the reflection principle.
\end{proof}

Our goal is now to apply Lemma \ref{lemma:lemma_31}. Because this only works in finite dimensions, we need to approximate $\bbP$ weakly by finite-dimensional Gaussians. This is possible by e.g. Theorem 3.1 in \cite{BeSchSe25}. In what follows $\mu \in \{\mu^m : m \ge 1 \}$ will always reference an approximation to $\bbP$. Unless otherwise specified, the following results hold uniformly for all $\mu = \mu^m$ with $m$ large enough. Moreover, note that the sets $B_i$ have $\bbP$-boundary measure $0$, which implies that for $m$ large enough, $\mu^m(B_j) \ge 1-2^{-10^i}+\epsilon(m)$. We will omit this additional $\epsilon$ for the sake of readability. As we have noted earlier, it suffices to bound
$$\bbE^{\bbPw_{\alpha,T,\gamma,\xi}} \lt[ \big|  \overline{s}_{0} ^{2^i}  \big|^2 \rt]$$
for every $i \le t$. We therefore fix $i \le t$ for the remainder of this Section and show that each variance is bounded by essentially $2^{i(\xi-1-\gamma/2)}$.

Recall that 
$$\bbE^{\bbP} \lt[ \big|  \overline{s}_{0} ^{2^i}  \big|^2 \rt] = O(2^i)$$
due to the normalization in the definition of $\overline{s}_{0} ^{2^i}$.
Let us now write
$$\mu = (1-\delta )\rho + \delta \rho',$$
where the decomposition stems from Lemma \ref{lemma:lemma_31} applied to the set $B_i$. Therefore, $\delta \le 2^{-10i}$ and $\supp(\rho) \subseteq 64 B_i$. Moreover, $\rho' \preceq \mu^{\times 2}$. 
Let us now write 
$$\Tilde{\mu}_{\alpha,T,\xi,\gamma} = (1-w_1) \Tilde{\rho}_{\alpha,T,\xi,\gamma} + w_1 \Tilde{\rho}'_{\alpha,T,\xi,\gamma}.$$
\begin{proposition}
    The inequality
    $$w_1 \le \delta$$
    holds.
\end{proposition}
\begin{proof}
    We can explicitly calculate
    $$(1-w_1) = (1-\delta) \frac{Z_{\rho}}{Z_\mu}.$$
    Because $\rho \preceq \mu$ we have $Z_\rho / Z_\mu \ge 1$ (as the function we integrate over is the (uniformly bounded) limit of symmetric and quasi-concave functions). Therefore, the claim is immediate.
\end{proof}
Define the measure
$$\mu' (\de x) \propto \exp \lt( - \frac{\alpha (2^{i/2}\sqrt{20\log(2^{2i})})^{\gamma-2}}{2^{i(\xi-2)}} \Big| \overline{s}_{0} ^{2^i}\Big|^2  \rt) \mu(\de x).$$
\begin{proposition}
    \label{prop:simple_domination}
    It holds that
    $$\Tilde{\rho}_{\alpha,T,\xi,\gamma} \preceq \mu'.$$
\end{proposition}

\begin{proof}
    Simply writing down the density, we find
    \begin{equation}
        \nn
        \begin{split}
            \frac{\de \Tilde{\rho}_{\alpha,T,\xi,\gamma} }{\de \mu' } \propto \exp \lt(- \alpha \int_0 ^T \de t \int_0 ^T \de s \frac{|x_t - x_s|^\gamma}{1+|t-s|^\xi} +  \frac{\alpha (2^{i/2}\sqrt{20\log(2^{2i})})^{\gamma-2}}{2^{2i(\xi-2)}} \Big| \overline{s}_{0} ^{2^i}\Big|^2  \rt).
        \end{split}
    \end{equation}
    On the support of $\rho$ we have that
    $$|x_t - x_s|^{\gamma-2} \ge \lt(2^{i/2}\sqrt{20\log(2^{2i})} \rt)^{\gamma-2}.$$
    Moreover, the time decay is at worst $2^{i(\xi-2)}$. This, together with $\rho \preceq \mu$ shows 
    $$\Tilde{\rho}_{\alpha,T,\gamma,\xi} \preceq \Tilde{\mu}_{ \frac{\alpha (2^{i/2}\sqrt{20\log(2^{2i})})^{\gamma-2}}{2^{i(\xi-2)}},2^i,0}.$$
    The result then follows from Proposition \ref{prop:quadratic_form_domination}.
\end{proof}

\begin{proof}[Proof of Theorem \ref{thm:sub_diffusive_regime}, (1) and (2)]
As we argued before, it suffices to bound 
$$\bbE^{\bbPw_{\alpha,T,\gamma,\xi}} \lt[ \big|  \overline{s}_{0} ^{2^i}  \big|^2 \rt]$$
for every $i \le t$.
Write $f_n(x) = \min(n,|\overline{s}_{0} ^{2^i}  |^2)$. Note that $x \mapsto \overline{s}_{0} ^{2^i}$ is a linear function on path space. Therefore, $f_n$ is a symmetric,quasi-convex, continuous and bounded function on path space.
By what we have just shown,
\begin{equation}
    \nn
    \begin{split}
        \bbE^{\Tilde{\mu}_{\alpha,T,\gamma,\xi}} \lt[   f_n \rt] &\le \bbE^{\Tilde{\rho}_{\alpha,T,\gamma,\xi}} \lt[ \big|  f_n  \big|^2 \rt] + 2^{-10i} \bbE^{\Tilde{\rho'}_{\alpha,T,\gamma,\xi}} \lt[ \big|  f_n  \big|^2 \rt] \\
        &\le \bbE^{\mu'} \lt[ \big|  f_n  \big|^2 \rt] + 2^{-8i}.
    \end{split}
\end{equation}
First letting $m \to \infty$ and then $n \to \infty$ therefore yields the bound
$$\bbE^{\bbPw_{\alpha,T,\gamma,\xi}} \lt[ \big|  \overline{s}_{0} ^{2^i}  \big|^2 \rt] \le \bbE^{\bbP'} \lt[ \big|  \overline{s}_{0} ^{2^i}  \big|^2 \rt] + 2^{-8i}$$
Now, a simple computation shows the existence of a constant $C>0$ so that
$$\frac{\alpha (2^{i/2}\sqrt{20\log(2^{2i})})^{\gamma-2}}{2^{i(\xi-2)}} \ge Ci^{-2}\alpha 2^{i(\gamma/2 + 1 -\xi)}.$$
In particular,
$$\bbE^{\bbP'} \lt[ \big|  \overline{s}_{0} ^{2^i}  \big|^2 \rt] \le \frac 1 C \alpha^{-1} i^2 2^{-i(\gamma/2 + 1 -\xi)}.$$
This shows the desired claim by setting $i=t$, which is the worst index for $\xi \in (1+\gamma/2,2+\gamma/2)$. In the case $\xi < 1+ \gamma/2$, we use the same approach to bound   
$$\bbE^{\bbPw_{\alpha,T,\gamma,\xi}} \lt[ \big|  \overline{s}_{0} ^{2^i}  \big|^2 \rt] \le \bbE^{\bbP'} \lt[ \big|  \overline{s}_{0} ^{2^i}  \big|^2 \rt] + 2^{-8i} \le C2^{-i\delta}$$
for some $\delta > 0$ and $C>0$. In particular, the sum over all scales of averaged increments is now a convergent series, which implies via the first variance bound in \eqref{equ:simple_variance_inequality} that the variance stays bounded. This shows the claim.
\end{proof}

\begin{proof}[Proof of Theorem \ref{thm:sub_diffusive_regime} (3)]
For the last assertion, we simply copy the proof in the beginning of Section \ref{sec:proof_of_thm12}. By restricting Brownian motion on $[0,T]$ to the set
   $$\lt\{\sup\limits_{s,t \in[0,T]} |x_t-x_s| \le \sqrt{20T\log T} \rt\}$$
   and doing all finite-dimensional approximations again,
   we find that 
   $$\bbE^{\bbPw_{\alpha,T,\gamma,\xi}}[|x_T|^2] \le \bbE^{\bbPw_{T^{-\xi+\gamma/2-1}(20\log(T))^{\gamma-2},T,0}}[|x_T|^2]+T^{-10}.$$
   A re-scaling argument exactly as in the beginning of Section \ref{sec:proof_of_thm12} shows the claim.
\end{proof}

\section{Proof of Theorem \ref{thm:main_thm_2}: recursive Gaussian confinement}
\label{sec:gaussian_bootstrap}
Recall that by FKG-GCI, in order to prove Theorem \ref{thm:main_thm_2}, it suffices to show
$$\bbE^{\bbPw_{\alpha,T,\gamma,\xi}}[|x_T|^2] \le T^{-1}+ 64^{2-\gamma}\frac{2}{\alpha\gamma}T^{\frac{2}{\gamma}(\xi-2)+ 10\log_T\log(T)^2/\gamma}.$$

As our approach relies on iterating and improving localization bounds, we start by studying the type of maps we iterate.
\begin{proposition}
    \label{prop:fixpoint}
    Fix $C>0$ and $d<1$. Let
    $$h: \bbR \to \bbR, x \mapsto C+ dx.$$
    Then, $h$ has a unique fixed point at $x^* = C/(1-d)$.
    Moreover, with $c= |\log(d)|$
    $$|x^* - h^n(1)| \le \frac{C}{1-d}\exp(-n c).$$
    Here, $h^n$ is defined recursively for $n \in \bbN$ by $h^n (x) = h( h^{n-1}(x))$ and $h^1 := h$.
\end{proposition}

\begin{proof}
    The fact that $x^*$ is the unique fixed point is obvious. Exponential speed of convergence can be seen by writing
    $$h^n (1) = d^n + C \sum\limits_{k=0} ^{n-1} d^k.$$
    This yields via the geometric series
    $$|h^n(1) - x^*| \le d^n \frac{C}{1-d}.$$
    The claim follows directly.
\end{proof}

\begin{proposition}
    \label{prop:gaussian_measure_variance_bound}
    Fix $\beta > 0$. Let $\mu_\beta$ be a Gaussian path measure satisfying $\mu_\beta \preceq \bbP$ as well as $\sup_{0 \le s \le T}\bbE^{\mu_\beta}[x_s^2] \le T^\beta$. Then,
    $$\mu_\beta \lt( \sup\limits_{0 \le s \le T} |x_s| \ge T^{\beta/2}\sqrt{40 \log(T)} \rt) \le 4 T^{-9}.$$
\end{proposition}

\begin{proof}
    Abbreviate $C = T^{\beta/2}\sqrt{40 \log(T)}$. Note the set inclusion
    $$\lt\{ \sup\limits_{0 \le s \le T} |x_s| \ge C \rt\} \subseteq \lt\{ \sup\limits_{j \in [0,T] \cap \bbN} |x_j| \ge C/2 \rt\} \cup \lt( \bigcup\limits_{j=0} ^{T-1} \lt\{ \sup\limits_{j\le s < t \le j+1} |x_t - x_s| \ge C/2 \rt\} \rt).$$
    We bound all sets individually and perform a union bound.
    First, by (GCI), standard tail bounds for Gaussians and the definition of $C$ it holds that
    $$\mu_\beta \lt( \sup\limits_{j \in [0,T] \cap \bbN} |x_j| \le C/2 \rt) \ge \prod\limits_{j=0} ^T \mu_\beta \lt( |x_j| \le C/2 \rt) \ge (1- T^{-10})^T \ge 1-T^{-9}.$$
    Similarly,
    \begin{equation}
        \nn
        \begin{split}
            \mu_\beta \lt( \sup\limits_{j\le s < t \le j+1} |x_t - x_s| \ge C/2 \rt) &\le \bbP \lt( \sup\limits_{0\le s < t \le 1} |x_t - x_s| \ge C/2 \rt) \\
            &\le 2 \bbP( |x_1| \ge C/2) \\
            &\le 2\exp(-T^\beta) T^{-10}.
        \end{split}
    \end{equation}
    Here, the first inequality follows from (GCI) and the second inequality via the reflection principle of Brownian motion. The overall claim follows by union bounding.
\end{proof}
In what follows we abbreviate $C(T) = \sqrt{40 \log(T)}$ and $D(T) =  \log_T(\log(T) + \log(T)^2)$.
The previous Proposition has the following implication: with very high probability and up to logarithmic corrections, 
$$\bbPw_{\alpha,T,\gamma,\xi} \preceq \bbPw_{\alpha T^{(\gamma-2)/2},T,\xi}.$$
This follows by bounding the probability of fluctuations of $\bbPw_{\alpha,T,\gamma}$ by that of Brownian motion and estimating on the set $\{ \sup_{0 \le s < t \le T} |x_t - x_s| \le T^{1/2}\}$
\begin{equation}
    \nn
    \begin{split}
        \int_0 ^T \de t \int_0 ^T \de s \frac{|x_t - x_s|^\gamma}{1+|t-s|^\xi} &= \int_0 ^T \de t \int_0 ^T \de s |x_t - x_s|^{\gamma-2} \frac{|x_t - x_s|^2}{1+|t-s|^\xi} \\
        &\ge T^{(\gamma-2)/2} \int_0 ^T \de t \int_0 ^T \de s \frac{|x_t - x_s|^2}{1+|t-s|^\xi}.
    \end{split}
\end{equation}
An application of Theorem \ref{thm:main_thm_1} and (GCI) yields
$$\sup\limits_{s \le T} \bbE^{\bbPw_{\alpha T^{(\gamma-2)/2},T,\xi}} [|x_s|^2] \le 1+ \alpha^{-1}T^{(2-\gamma)/2+\xi-2}(\log(T) + \log(T)^2).$$
To leading order this variance bound behaves as $T^{\xi - 1 - \gamma/2}$, which already yields sub-diffusive behaviour as soon as $\xi < 2+\gamma/2$. Let us take the new exponent to be $\beta = \xi - 1 - \gamma/2 < 1$. Ignoring the log-factors again for a second, we find that the increments under $\bbPw_{\alpha,T,\gamma,\xi}$ are not only contained in $T^{1/2}$, but even in $T^{\beta/2}$ with (still) very high probability. This follows  from another application of Proposition \ref{prop:gaussian_measure_variance_bound} to the measure $\bbPw_{\alpha T^{(\gamma-2)/2},T,\xi}$. Iterating this procedure yields an exponent of $T$ getting close to the fix-point of the map
\begin{equation}
    \label{equ:recursive_variance_improvement}
     \beta \mapsto \xi-2 + \beta - \frac{\beta\gamma}{2},
\end{equation}
see Proposition \ref{prop:fixpoint}. The iteration we perform therefore does not yield different regimes of sub-diffusivity, but rather improved rates. Indeed, let us compare the iterated bound and the one-step bound. While it is not true for any combination of $\gamma >0$ and $\xi >0$ that
$$\frac 2 \gamma (\xi - 2) < \xi-1-\frac \gamma 2,$$
simple algebra shows that the inequality holds whenever $\xi - 1-\gamma/2 < 1$. As this is exactly the range of parameters that we are interested in, iterating always improves the exponent. Moreover, note that a first iteration only yields an exponent smaller than $1$ whenever $\xi < 2 + \gamma/2$. Therefore, our approach does not work outside of this regime.

In what follows we will implement our iteration strategy rigorously. 
We define the set
$$R(\beta) := \lt\{ \sup\limits_{0 \le s < t \le T} |x_t - x_s| \le \beta \rt\},$$
which will serve as our choice for $K$ in Lemma \ref{lemma:lemma_31}. 
In order to implement the desired recursion we will rely heavily on Lemma \ref{lemma:lemma_31}. As noted in the previous section, this Lemma is only valid for finite-dimensional Gaussians. Therefore, we will approximate $\bbP$ weakly by a sequence of finite-dimensional, centred Gaussian measures $(\mu^m)_{m \ge 1}$ on $C \lt( [0,T] ; \bbR^d \rt)$. Let us also note that the set $R(\beta)$ has $\bbP$-boundary measure $0$ for any $\beta >0$. It follows that, as $m \to \infty$,
$$\Tilde{\mu}^m_{\alpha,T,\xi} \lt( R(\beta) \rt) \to \bbPw_{\alpha,T,\xi} \lt( R(\beta) \rt).$$
Theorem \ref{thm:main_thm_1} together with Proposition \ref{prop:gaussian_measure_variance_bound} yield with $\eta = C(T)\sqrt{D(T)T^{(\xi-2)}}$ that
$$\Tilde{\mu}^m_{\alpha,T,\xi} \lt( R(\eta) \rt) \ge 1- 4T^{-9}- \Tilde{\epsilon}(m).$$
Due to $\Tilde{\epsilon}(m) \to 0$ as $m \to \infty$ we will omit this additional factor for sake of convenience. Alternatively, one could replace $c$ by $c/(1+\delta)$, where $\delta$ can be arbitrarily small (by choosing $m$ arbitrarily big). 
In what follows $\mu \in \{\mu^m : m \ge 1 \}$ will always reference an approximation to $\bbP$. Unless otherwise specified, the following results hold uniformly for all $\mu = \mu^m$ with $m$ large enough. Take
for $j \ge 1$
$$S_j = (V_{j-1})^{2-\gamma}T^{\xi-2+ D(T)},V_j = (S_j) ^{1/2} C(T),$$
where $S_0 = T$. Note that if $S_j = T^\beta$ for some $\beta > 0$, then
$$S_{j+1} = T^{\xi-2 + D(T) + 2\log_T (C(T))/ \gamma + \beta - \beta\gamma/2}.$$
In other words, the exponent behaves like $\eqref{equ:recursive_variance_improvement}$ but with $T$-dependent corrections. The corrections arise due to the $\log$ factors in the bound of Theorem \ref{thm:main_thm_1} and the $\log(T)$ factor necessary in Proposition \ref{prop:gaussian_measure_variance_bound} to guarantee that we only lose $T^{-9}$ mass by restricting to $V_j$. In any case, we have that one iteration (i.e. moving from $S_j$ to $S_{j+1}$) maps
\begin{equation}
    \label{equ:exponent_evolution}
    \log_T (S_j) = \beta \to \xi - 2 + 2\log_T C(T)/\gamma + \beta - \beta \gamma/2 + D(T) = \log_T(S_{j+1}).
\end{equation}
In particular, the exponent evolves as a linear map in $\beta$. Therefore, Proposition \ref{prop:fixpoint} can be applied to deduce how close we come to the fixed point, which in our case evaluates to
$$\frac 2 \gamma (\xi-2 + \log_TC(T)) + D(T).$$

\begin{lemma}
    \label{lemma:recursive_decomposition_2}
    There exists a decomposition of $\mu$ into probability measures $\nu_{T^7}$ and $\nu'_j$ ($j \le T^7$) such that
    $$\mu = \prod\limits_{s=1}^{T^7}  (1-\delta_s) \nu_{T^7} + \sum\limits_{j= 1} ^{T^7} \delta_j\prod\limits_{s=1}^{j-1}  (1-\delta_s) \nu'_j.$$
    Moreover, the decomposition has the following properties:
    \begin{itemize}
        \item $\nu_{T^7} \preceq \mu$;
        \item $\supp(\nu_{T^7}) \subseteq 64 R(V_{T^7})$;
        \item $\delta_j \le 4 T^{-9}$;
        \item $(\Tilde{\nu_j}')_{\alpha,T,\gamma,\xi} \preceq \mu^{\times 2}$;
        \item $\supp(\nu'_j) \subseteq 64 R(V_{j-1})$.
    \end{itemize}
\end{lemma}
\begin{proof}
    As discussed in the previous paragraph, by definition of $R(V_1)$
    $$\mu \lt( R(V_1) \rt) \ge 1- 4T^{-9}.$$
     An application of Lemma \ref{lemma:lemma_31} allows us to write the mixture decomposition
    $$\mu = (1-\delta_1) \rho_1 + \delta_1 \overline{\rho}_1,$$
    where $\supp(\rho_1) \subseteq R(64 V_1)$. Let us take $\nu_1 := \rho_1, \nu' _1 := \overline{\rho}_1$
    and abbreviate $\beta = \frac{\alpha\gamma}{2} 64^{\gamma-2}$. For convenience we choose $\alpha$ so that $\beta >1$. One can avoid this by carrying around another additional multiplicative factor. We then write
    $$\de \rho_1 = \frac{\de \rho_1}{ \de \Tilde{\mu}_{\beta (V_1 )^{\gamma-2},T,\xi}} \de \Tilde{\mu}_{\beta (V_1) ^{\gamma-2},T,\xi}.$$
By Theorem \ref{thm:main_thm_1} we find that 
$$\sup\limits_{s \le t} \bbE^{\Tilde{\mu}_{\beta (V_1 )^{\gamma-2},T,\xi}} [|x_s|^2] \le \beta^{-1}S_2.$$
Therefore,  Proposition \ref{prop:gaussian_measure_variance_bound} yields
$\Tilde{\mu}_{\beta (V_1) ^{\gamma-2},T,\xi}( V_2 ) \ge 1- 4 T^{-9}.$
An application of Lemma \ref{lemma:lemma_31} is now justified, so that
$$\Tilde{\mu}_{\beta V_1 ^{\gamma-2},T,\xi} = (1-\delta_2) \rho_2 + \delta_2 \overline{\rho}_2,$$
where the good part $\rho_2$ is supported on $R(64 V_2)$.
This leads us to define
$$\nu_2 \propto \frac{\de \rho_1}{ \de \Tilde{\mu}_{\beta (V_1) ^{\gamma-2},T,\xi}} \rho_2 , \:\:\:\:\:\:\: \nu'_2 \propto \frac{\de \rho_1}{ \de \Tilde{\mu}_{\beta (V_1) ^{\gamma-2},T,\xi}} \overline{\rho}_2 .$$
We can now continue inductively. Assuming $\nu_n$ is defined as above, write
$$\de \nu_{n} = \frac{\de \nu_{n}}{ \de \Tilde{\mu}_{\beta (V_n) ^{\gamma-2},T,\xi}} \de \Tilde{\mu}_{\beta (V_n) ^{\gamma-2},T,\xi}.$$
Once again, a combination of Theorem \ref{thm:main_thm_1} and Proposition \ref{prop:gaussian_measure_variance_bound} yields
$$\Tilde{\mu}_{\beta (V_n) ^{\gamma-2},T,\xi}(V_{n+1}) \ge 1- 4T^{-9}.$$
Define $\nu_{n+1}$ by
$$\nu_{n+1} \propto \frac{\de \nu_n}{ \de \Tilde{\mu}_{\beta (V_n) ^{\gamma-2},T,\xi} } \rho_{n+1}, \:\:\:\:\:\:\: \nu'_{n+1} \propto \frac{\de \nu_n}{ \de \Tilde{\mu}_{\beta (V_n) ^{\gamma-2},T,\xi} } \overline{\rho}_{n+1},$$
where $\rho_{n+1}$, $\overline{\rho}_{n+1}$ are defined via Lemma \ref{lemma:lemma_31} by
$$\Tilde{\mu}_{\beta (V_n) ^{\gamma-2},T,\xi} = (1-\delta_{n+1}) \rho_{n+1} + \delta_{n+1} \overline{\rho}_{n+1}.$$
In particular, 
\[
\supp(\nu_{n+1}) \subseteq R(64 V_{n+1})
\]
and $\nu_{n+1} \preceq \mu$ as well as  $\delta_{n+1} \le 4T^{-9}$. The fact that $\supp(\nu'_{n+1}) \subseteq 64 R(V_{n})$ follows straight from the definition.
Finally, we show that
$$(\Tilde{\nu}'_j)_{\alpha,T,\gamma,\xi} \preceq \mu^{\times 2}.$$
Writing out the definition of $(\Tilde{\nu}'_j)_{\alpha,T,\gamma,\xi}$ we find 
\begin{equation}
    \nn
    \begin{split}
        \frac{\de (\Tilde{\nu}'_j)_{\alpha,T,\gamma,\xi} }{\de \mu^{\times 2}} &\propto \exp \lt(  \alpha \int_0 ^T \de t \int_0 ^T \de s -\frac{|x_t - x_s|^\gamma}{1+|t-s|^\xi} + \frac \gamma 2(64V_{j-1})^{\gamma-2}  \frac{|x_t - x_s|^2}{1+|t-s|^\xi} \rt) \frac{\de \nu_{j-1}}{\de \mu} \frac{\de\overline{\rho}_j}{\de \mu^{\times 2}} \\
        &\preceq \exp \lt(  \alpha \int_0 ^T \de t \int_0 ^T \de s -\frac{|x_t - x_s|^\gamma}{1+|t-s|^\xi} + \frac \gamma 2(64V_{j-1})^{\gamma-2}  \frac{|x_t - x_s|^2}{1+|t-s|^\xi} \rt) \one_{R(64V_{j-1})}.
    \end{split}
\end{equation}
Here, the domination follows from the fact that $\de \nu_{j-1} / \de \mu$ is the uniformly bounded limit of quasi-concave functions; the same holds for $\de \overline{\rho}_j / \de \mu^{\times 2}$. On $R(64V_{j-1})$ we have that the double integral in the last density is also the uniformly bounded limit of symmetric and quasi-concave functions, which shows the final claim.
\end{proof}
We want to point out that $\delta_j$ depends on $m$, since by definition it is a weight that depends on $\mu = \mu^m$. This does not pose a problem: we only need an upper bound on $\delta_j$ to obtain our estimates.
We apply Lemma \ref{lemma:recursive_decomposition_2} to find the decomposition
$$\Tilde{\mu}_{\alpha,T,\gamma,\xi} = w^* (\Tilde{\nu}_{T^7})_{\alpha,T,\gamma,\xi} +   \sum\limits_{j=1} ^{T^7} w_j (\Tilde{\nu}'_j)_{\alpha,T,\gamma,\xi}$$
for some non-negative numbers $w^*,w_1,\dots,w_{T^7}$ satisfying $ w^* +\sum_{j \le T^7} w_j  = 1$. In particular, we have that
\begin{equation}
    \nn
    \begin{split}
        w^* &= \prod\limits_{j=1} ^{T^7}(1-\delta_j) \frac{ \int \nu_{T^7}(\de x) \exp \lt( -\alpha \int_0 ^T \de t \int_0 ^T \de s \frac{| x_t - x_s|^{\gamma}}{1+|t-s|^{\xi}} \rt) }{ \int \bbP(\de x) \exp \lt( -\alpha \int_0 ^T \de t \int_0 ^T \de s \frac{| x_t - x_s|^{\gamma}}{1+|t-s|^{\xi}} \rt) } \\
        &\ge \prod\limits_{j=1} ^{T^7}(1-\delta_j) 
        \ge (1-4T^{-9})^{T^7} 
        \ge 1-4T^{-2}.
    \end{split}
\end{equation}
The first inequality follows by (GCI) and the fact that $\nu_{T^7} \preceq \mu$. The second inequality follows from $\delta_j \le 4T^{-9}$. In conclusion,
\begin{equation}
    \label{equ:bad_parts_low_weight}
    \sum\limits_{j=1} ^{T^7} w_j \le 4T^{-2}.
\end{equation}

\begin{proof}[Proof of Theorem \ref{thm:main_thm_2}]

Write $f_n(x) := \min \{ x^2,n\}$. Note that $f_n$ is increasing, symmetric, continuous and bounded. In what follows we again abbreviate $\beta = \frac{\alpha\gamma}{2} 64^{\gamma-2}$.
By the decomposition obtained in Lemma \ref{lemma:recursive_decomposition_2} and what was just discussed
\begin{equation}
    \nn
    \begin{split}
        \bbE^{\Tilde{\mu}^m_{\alpha,T,\gamma,\xi}} [f_n(x_T)] &= \sum\limits_{j=1} ^{T^7} w_j \bbE^{(\Tilde{\nu}'_j)_{\alpha,T,\gamma,\xi}} [f_n(x_T)] +  w^* \bbE^{(\Tilde{\nu}_{T^7})_{\alpha,T,\gamma,\xi}} [f_n(x_T)]\\
        &\le 4T^{-2} \bbE^{(\mu^m) ^{\times 2}}[f_n(x_T)] + (1-4T^{-2}) \bbE^{\Tilde{\mu}^m_{\beta (V_{T^7}) ^{\gamma-2},T,\xi}}[f_n(x_T)].
    \end{split}
\end{equation}
The inequality is justified by $(\Tilde{\nu}'_j)_{\alpha,T,\gamma,\xi} \preceq (\mu^m) ^{\times 2}$, equation \eqref{equ:bad_parts_low_weight} and $ (\Tilde{\nu}_{T^7})_{\alpha,T,\gamma,\xi} \preceq \Tilde{\mu}^m_{\beta (V_{T^7}) ^{\gamma-2},T,\xi}$.
It is now possible to take $m \to \infty$, which takes $\mu^m \to \bbP$. This follows from continuity and boundedness of all functions we integrate over and weak convergence. 
Taking now $n \to \infty$ and exchanging limit and integration thanks to the monotone convergence Theorem, the inequality
$$\bbE^{\bbPh_{\alpha,T,\gamma,\xi}} [|x_T|^2]  \le 4T^{-2} \bbE^{\bbP ^{\times 2}}[|x_T|^2] + (1-4T^{-2}) \bbE^{\bbPw_{\beta (V_{T^7})^{\gamma-2},T,\xi}}[|x_T|^2].$$
is deduced.
Finally,
$$4T^{-2} \bbE^{\bbP ^{\times 2}}[|x_T|^2] \le 8 T^{-1}$$
and an application of Theorem \ref{thm:main_thm_1} yields
$$\bbE^{\bbPw_{\beta (V_{T^7})^{\gamma-2},T,\xi}}[|x_T|^2] \le S_{T^7}.$$
By the relation \eqref{equ:exponent_evolution} and Proposition \ref{prop:fixpoint} we know that there exists $c>0$ so that
$$\Big| \log_T S_{T^7} - \frac 2 \gamma (\xi-2 + \log_TC(T) + D(T)) \Big| \le \exp(-T^7 c).$$
Combining yields
$$\bbE^{\bbPw_{\beta (V_{T^7})^{\gamma-2},T,\xi}}[|x_T|^2] \le \beta^{-1}T^{ \frac 2 \gamma (\xi-2 + \log_TC(T) + D(T)) + \exp(-T^7 c) },$$
which is what was claimed.
\end{proof}

\section{Proof of Theorem \ref{thm:main_thm_3}: Hierarchical decomposition}
\label{sec:hierarchical_decomposition}

We begin this Section by discussing some differences and similarities to the approach presented in the previous Section. On a technical level, we will again approximate Brownian motion by a finite-dimensional Gaussian and then iteratively decompose the approximation into sets with better and better confinement (compare Lemma \ref{lemma:recursive_decomposition_1} and Lemma \ref{lemma:recursive_decomposition_2}). The decomposition is, however, a different one. In the previous Section, we replaced the time-decay by its worst case estimate and restricted the path to have bounded fluctuations. We iterated this procedure, obtaining better and better concentration results for almost all mass. Here, we provide a different approach: instead of recursively iterating 'global confinement', we recursively improve spatial localization. In a first step, we bound the fluctuation of level $0$ increments (which in our language are just $\sigma_j$ for $0 \le j \le T-1$). This gives us the possibility to obtain localization results for level $1$ averaged increments. Level $1$ averaged increments together with level $0$ increments can then be used to pointwise bound fluctuations on twice the length scale (see Proposition \ref{prop:r_bound}). Basically, we are using localization properties of the previous level to obtain localization properties for the next level and repeat. We then observe that if $\xi<2$ and $\alpha = O(\log(T))$, this iteration leads to bounded variances for all averaged increments! Because the variance of the path at time $T$ is the sum of logarithmically many averaged increments, this implies logarithmic fluctuations at most. We include all details once again (which we believe leads to greater clarity, at the cost of some repetition). 

Fix $\gamma,\xi$ and $\alpha$.
In this section the following sequence of numbers will be used throughout. At first, define $r_0 := \sqrt{20}$. To determine the next numbers in the sequence, the following recursive relations are used:
$$ r_{n+1} := 256\sum\limits_{k=0} ^n A_n, \:\:\:\:\:\:\:\:\: A_{n} := r_n ^{1- \frac{\gamma}{2}} 2^{\frac{n}{2} (\xi-2)}.$$ 
These quantities will occur naturally when iteratively confining averaged increments; see e.g. \eqref{equ:density_rewrite}. One should think of $A_n$ as the standard deviation of the $n$-th level averaged increments if one knows that the path has pointwise fluctuations of at most $r_n$. Similarly, we can bound with high probability the pointwise fluctuations of the path at scale $2^{n+1}$ by standard deviations of the averaged increments at smaller scales, similarly to \eqref{equ:variance_decomposition}. This is what $r_{n+1}$ represents. 
For future use we abbreviate
$$\beta = (256\sqrt{20}t)^{2-\gamma}.$$
Now, define
\begin{equation}
    \label{equ:mu_hierarchical}
    \Lambda_{\alpha,j,T}(\mu)(\de x) \propto \exp\lt(- \frac{\alpha}{2}\sum\limits_{l= 1} ^{j} A_l ^{-2} \sum\limits_{s = 0} ^{T - 2^l} \big| \overline{s}_{s 2^l} ^{2^l} \big|^2 \rt) \mu(\de x).
\end{equation}
$\Lambda_{\alpha,j,T}(\mu)$ will serve a similar purpose as $\bbP_{\alpha,T,\xi}^{\mathbf{hier}}$ (recall equation \eqref{equ:hierarchical_bm}). Note, however, that $\Lambda_{\alpha,j,T}(\mu)$ only has interactions up to level $j$ (i.e. block length $2^j$), \textbf{but uniformly over all blocks}. The uniformity is really necessary, because the presented approach relies on path-wise estimates.
 We will again take $T= 2^t$ for some $t \in \bbN$ as in section \ref{sec:hierarchical_path_measures1}. In case $j=t$, we write
$$\Lambda_{\alpha,j,T}(\mu) = \mu^{\mathbf{hier}}_{\alpha,T}$$
in analogy to \eqref{equ:hierarchical_bm}.
One central ingredient in the proof of Theorem \ref{thm:main_thm_2} is a variance bound similar to the one deduced in Lemma \ref{lemma:simple_density_replacement}.
\begin{lemma}
    \label{lemma:main_variance_inequality}
    Choose $T= 2^t$ for some $t \in \bbN$. Then,
    \begin{equation}
        \label{equ:main_variance_inequality}
        \bbE^{\bbPh_{\alpha \beta,T,\gamma,\xi}} [|x_T|^2] \le (1-T^{-7})\bbE^{\bbP^{\mathbf{hier}} _{\alpha,T}} [|x_{T}^2] + 2 T^{-6}.
    \end{equation}
\end{lemma}
We recall that it suffices to show \eqref{equ:main_variance_inequality} for $\bbPw_{\alpha\beta,T,\gamma,\xi}$, where $\gamma$ is determined by $f$ through condition \eqref{equ:qc_condition}. This was verified in Section \ref{sec:gci}.

 For every $ \bbN \ni j \le t$ define the random variable
$$R_{u,j} (x) := \sup\limits_{u \le s,t \le u+2^j} | x_t - x_s |, \:\:\:\:\:\:\: R_{j}(x) := \max\limits_{u = 0,2^j,\dots,2^t - 2^j}R_{u,j} (x).$$
Additionally, we take for $n \ge 0$
$$C_n := \bigcap\limits_{s=0} ^{2^t - 2^n}  \lt\{ \overline{s}_{2^n s} ^{2^n} \le \sqrt{20t} A_n \rt\}.$$
In words, $R_j$ determines the fluctuations of a a path along a block of size $2^j$, whereas $C_j$ is the set on which the increments of the path averages at level $j$ are bounded by an enlarged version of $A_j$ (cf. equation \eqref{equ:mu_hierarchical}). Let us also note that there is a $\log(T)$-dependent blow-up factor present in the definition of $C_n$. This enlargement is necessary due to the fact that we want these sets to have almost all mass as $T$ becomes large. This additional factor will effectively reduce the coupling constant $\alpha$ to $O(\alpha/ \sqrt{20 t})$. To counteract this, we need to choose $\alpha$ dependent on $T$.
\begin{proposition}
    \label{prop:cs_have_high_prob}
    Let $k \le t-1$. It holds that 
    $$\Lambda_{k,T}(\bbP) \lt( \bigcap\limits_{s=0} ^k C_s \rt) \ge 1- T^{-8}.$$
\end{proposition}
\begin{proof} 
    The estimate is a union bound over the $C_j$ (there are at most $2^t$ terms for each $C_j$) and noting that there are at most $t$ different $C_j$, which shows the claim.
\end{proof}

We will now show that increments on scale $2^{j+1}$ can be bounded by averaged increments of scale $2^{j}$. This is crucial, since we can obtain point-wise increment bounds on the path by controlling the (in our scenario) simpler averaged path increments. The proof follows along the same lines as the one in Lemma \ref{lemma:variance_decomposition} and the derivation of \eqref{equ:variance_decomposition}.

\begin{proposition}
    \label{prop:r_bound}
     Suppose that $\nu$ is a measure on $C([0,T];\bbR^d)$, which is supported (for $n \le t$) on
    $$ \bigcap\limits_{s=0} ^n C_s = \bigcap\limits_{s=0} ^n \bigcap\limits_{m=0} ^{T - 2^s} \Big\{\overline{s}_{2^s m} ^{2^s} \le  A_n \Big\}.$$ 
    Then,
    $$R_{n+1} \le  4(A_0 +\dots+A_n)$$
    $\nu$-almost surely.
\end{proposition}
\begin{proof}
    Fix $n \le t -1$ and take $u,v \in [0,2^{n+1}]$. A first simple estimate yields
    $$|x_u - x_v| \le |x_u - x_{2^n}| + |x_{2^n}- x_v|.$$
    We will bound the first term on the rhs; the other term can be treated analogously. 
    Let $u_1 := \floor{u}$. Then,
    $$|x_u - x_{2^n}| \le |x_u - \sigma_{u_1}|+|\sigma_{u_1} - \sigma_{2^{n-1}}| + | \sigma_{2^{n-1}} - x_{2^n}| \le 2 A_0 + |\sigma_{u_1} - \sigma_{2^{n-1}}|.$$
    The desired bound is shown if $u_1 = 2^n -1$. Else,
    there exists a unique $j \le n-2$ so that $u_1 \in [2j, 2j +2]$. Take $u_2 = 2j$ and continue to calculate
    $$|\sigma_{u_1} - \sigma_{2^{n-1}}| \le \big|\sigma_{u_1} - \frac 1 2 \sigma_{u_2} ^2 \big| + \big| \frac 1 2 \sigma_{u_2} ^2 - \frac 1 2 \sigma_{2^n - 2} ^2 \big| + \big| \frac 1 2 \sigma_{2^n -2} ^2 - \sigma_{2^n -1} \big| \le 2 A_1 + \big| \frac 1 2 \sigma_{u_2} ^2 - \frac 1 2 \sigma_{2^n - 2} ^2 \big|.$$
    If $u_2 = 2^n - 2$ we are done. Else, continue until either $u_j = 2^n - 2^j$ (for some $j <n$) or $j=n$. If $j=n$, then $u_j = 0$. This shows the claim.
\end{proof}

In order to implement the desired recursion we will rely heavily on Lemma \ref{lemma:lemma_31}. Because this Lemma is only valid for finite-dimensional Gaussians, we will approximate $\bbP$ weakly by a sequence of finite-dimensional, centred Gaussian measures $(\mu^m)_{m \ge 1}$ on $C \lt( [0,T] ; \bbR^d \rt)$. Let us also note that the sets $C_j$ have $\bbP$-boundary measure $0$. It follows that
$$\Lambda_{j,T}(\mu^m) \lt( \bigcap\limits_{s=0} ^j C_s \rt) \to \Lambda_{j,T}(\bbP) \lt( \bigcap\limits_{s=0} ^j C_s \rt).$$
In conjunction with Proposition \ref{prop:r_bound} we find
$$\Lambda_{j,T}(\mu^m) \lt( \bigcap\limits_{s=0} ^j C_s \rt) \ge 1- T^{-8} - \Tilde{\epsilon}(m).$$
Due to $\Tilde{\epsilon}(m) \to 0$ as $m \to \infty$ we will omit this additional factor for sake of convenience. Alternatively, one could replace $T^{-8}$ by $T^{-7}$. 
In what follows $\mu \in \{\mu^m : m \ge 1 \}$ will always reference an approximation to $\bbP$. Unless otherwise specified, the following results hold uniformly for all $\mu = \mu^m$ with $m$ large enough.

\begin{lemma}
    \label{lemma:recursive_decomposition_1}
    There exists a decomposition of $\mu$ into measures $\nu_t$ and $\nu'_j$ ($j \le t$) such that
    $$\mu = \prod\limits_{s=1}^t  (1-\delta_s) \nu_t + \sum\limits_{j= 1} ^t \delta_j\prod\limits_{s=1}^{j-1}  (1-\delta_s) \nu'_j.$$
    Moreover, this decomposition has the following properties:
    \begin{itemize}
        \item $\nu_t \preceq \mu$;
        \item $\supp(\nu_t) \subseteq 64 (C_0 \cap \dots \cap C_{t-1})$;
        \item $R_{t} \le 256 \sqrt{20t} (A_0 +\dots+A_{t-1})$ holds $\nu_t$-almost surely;
        \item $\delta_n \le T^{-8}$;
        \item $\Lambda_{\alpha,j-1,T}(\nu'_j) \preceq \mu^{\times 2}$;
        \item $\supp(\nu'_j) \subseteq 64 (C_0 \cap \dots \cap C_{j-2})$
    \end{itemize}
\end{lemma}
\begin{proof}
    As discussed in the previous paragraph, by definition of $C_0$
    $$\mu \lt( C_0 \rt) \ge 1- T^{-8}.$$
     An application of Lemma \ref{lemma:lemma_31} allows us to write the mixture decomposition
    $$\mu = (1-\delta_1) \rho_1 + \delta_1 \overline{\rho}_1,$$
    where $\supp(\rho_1) \subseteq 64 C_0$. Let us take $\nu_1 := \rho_1, \nu' _1 := \overline{\rho}_1$
    and write
    $$\de \rho_1 = \frac{\de \rho_1}{ \de \Lambda_{\alpha,1,T}(\mu)} \de \Lambda_{\alpha,1,T}(\mu).$$
As before, Proposition \ref{prop:cs_have_high_prob} yields
$$\Lambda_{\alpha,1,T}(\mu) ( C_0 \cap C_1) \ge 1- T^{-8}.$$
Therefore, an application of Lemma \ref{lemma:lemma_31} is justified so that
$$\Lambda_{\alpha,1,T}(\mu) = (1-\delta_2) \rho_2 + \delta_2 \overline{\rho}_2,$$
where the good part $\rho_2$ is now supported on $64 (C_0 \cap C_1)$.
This leads us to define
$$\nu_2 \propto \frac{\de \rho_1}{ \de \Lambda_{\alpha,1,T}(\mu)} \rho_2 , \:\:\:\:\:\:\: \nu'_2 \propto \frac{\de \rho_1}{ \de \Lambda_{\alpha,1,T}(\mu)} \overline{\rho}_2 .$$
We can now continue inductively. Assuming $\nu_n$ is defined as above, write
$$\de \nu_{n} = \frac{\de \nu_{n}}{ \de \Lambda_{\alpha,n,T}(\mu)} \de \Lambda_{\alpha,n,T}(\mu).$$
We apply Proposition \ref{prop:cs_have_high_prob}  for $m$ large enough as discussed to obtain.
$$\Lambda_{\alpha,n,T}(\mu) ( C_0 \cap C_1 \cap \dots \cap C_n) \ge 1- T^{-8}.$$
Define $\nu_{n+1}$ by
$$\nu_{n+1} \propto \frac{\de \nu_n}{ \de \Lambda_{\alpha,n,T}(\mu) } \rho_{n+1}, \:\:\:\:\:\:\: \nu'_{n+1} \propto \frac{\de \nu_n}{ \de \Lambda_{\alpha,n,T}(\mu) } \overline{\rho}_{n+1},$$
where $\rho_{n+1}$, $\overline{\rho}_{n+1}$ are defined via Lemma \ref{lemma:lemma_31} by
$$\Lambda_{\alpha,n,T}(\mu) = (1-\delta_{n+1}) \rho_{n+1} + \delta_{n+1} \overline{\rho}_{n+1}.$$
In particular, 
\[
\supp(\nu_{n+1}) \subseteq 64 (C_0 \cap\dots \cap C_{n})
\]
and $\nu_{n+1} \preceq \mu$ and $\delta_{n+1} \le T^{-8}$. An application of Proposition \ref{prop:r_bound} yields the $\nu_{n+1}$-almost sure desired bound on $R_{n+1}$. 

The fact that $$\Lambda_{\alpha,n,T}(\nu'_{n+1}) \preceq \mu^{\times 2}$$
is seen by a direct calculation. Indeed,
\begin{equation}
    \nn
    \begin{split}
        \Lambda_{\alpha,n,T}(\nu'_{n+1}) &\propto \exp\lt(- \frac \alpha 2\sum\limits_{l= 1} ^n A_l ^{-2} \sum\limits_{s = 0} ^{T - 2^l} \big| \overline{s}_{s 2^l} ^{2^l} \big|^2 \rt) \frac{ \de \nu_n}{ \de \Lambda_{\alpha,n,T}(\mu)} \overline{\rho}_{n+1} \\
        &\propto \frac{\de \nu_n}{ \de \mu} \overline{\rho}_{n+1} \\
        &\preceq \mu^{\times 2}. \qedhere
    \end{split}
\end{equation}
\end{proof}
We want to point out that $\delta_j$ depends on $m$, since by definition it is a weight that depends on $\mu = \mu^m$. This does not pose a problem: we only need an upper bound on $\delta_j$ to obtain our estimates (cf. equation \eqref{equ:weight_sum} and the proof of Lemma \ref{lemma:main_variance_inequality}).
Using the decomposition from the last Proposition it is possible to rewrite the reweighted measure $\Tilde{\mu}_{\alpha,T,\gamma,\xi}$  as
\begin{equation}
    \label{equ:reweighted_measure_decomposition}
    \Tilde{\mu}_{\alpha,T,\gamma,\xi} = \sum\limits_{j=1} ^t w(j) (\Tilde{\nu}_j ')_{\alpha,T,\gamma,\xi} + w^* (\Tilde{\nu_t})_{\alpha,T,\gamma,\xi},
\end{equation}

where 
$$w^* = \prod\limits_{s=1}^t (1-\delta_s) \frac{Z_{\nu_t}}{Z_\mu} \ge (1-T^{-8})^t \ge 1-T^{-7}.$$
Here, the first inequality follows as $\nu_t \preceq \mu$ and the function we integrate over is a uniformly bounded limit of symmetric and quasi-concave functions. We therefore find that
\begin{equation}
    \label{equ:weight_sum}
    \sum\limits_{j=1} ^t w(j) \le T^{-7},
\end{equation}
as $\sum_{j=1} ^t w(j) + w^* = 1$.

\begin{lemma}
    \label{lemma:density_replacement}
    The following inequalities hold
    \begin{enumerate}
        \item $(\Tilde{\nu}_t)_{\alpha \beta ,T,\gamma,\xi} \preceq \Lambda_{\alpha,t,T}(\mu)$;
        \item For all $s \le t$, $(\Tilde{\nu}'_s)_{\alpha \beta,T,\gamma,\xi} \preceq \mu^{\times 2}$.
    \end{enumerate}
\end{lemma}
\begin{proof}
    We start with showing
    $$(\Tilde{\nu}_t)_{\alpha\beta,T,\gamma,\xi} \preceq \mu_{\alpha,T} ^{\mathbf{hier}},$$
    which is just the first claimed inequality by recalling the introduced notation.
    Explicitly writing the density we find
    \begin{equation}
        \nn
        \begin{split}
            &\frac{ \de (\Tilde{\nu}_t)_{\alpha\beta,T,\gamma,\xi} }{ \de \mu_{\alpha,T} ^{\mathbf{hier}} } \\
            &\propto \exp \lt( -\alpha \beta \int_0 ^{T} \de t \int_0 ^{T} \de s \frac{| x_t - x_s|^{\gamma }}{1+|t-s|^{\xi}} + \frac{\alpha}{2 }\sum\limits_{l= 1} ^t A_l ^{-2} \sum\limits_{s = 0} ^{{T} - 2^l} \big| \overline{s}_{s 2^l} ^{2^l} \big|^2 \rt).
        \end{split}
    \end{equation}
    Recall from Lemma \ref{lemma:recursive_decomposition_1} that
    $$\supp(\nu_t) \subseteq \bigcap\limits_{s=1} ^t \lt\{  R_s \le 256 \sqrt{20 t} \sum\limits_{k=0} ^{s-1} A_k \rt\}.$$
    Moreover, recall also that
    $$A_j =  r_j ^{1- \frac \gamma 2} 2^{\frac j 2 (\xi-2)}.$$
    By definition we then have, for $s,t$ being in the same block of length $2^j$,
    \begin{equation}
        \label{equ:density_rewrite}
        \begin{split}
            \frac{|x_t - x_s|^{\gamma}}{1+ |t-s|^{\xi}} &= |x_t - x_s|^2 \frac{1}{|x_t - x_s|^{2- \gamma}} \frac{1}{1+ |t-s|^{\xi}} \\
            & \ge |x_t - x_s|^2 \frac{1}{2 r_j ^{2- \gamma} 2^{j\xi}\beta }  \\
            &= |x_t - x_s|^2 \frac
        1 { 2 A_j ^2 2^{2j} \beta}.
        \end{split}
    \end{equation}
    Together with Proposition \ref{prop:quadratic_form_domination} this shows 
    $$\beta\int_{0} ^{{2^{j-1}}} \de t \int_{{2^{j-1}}} ^{2^j} \de s \frac{|x_t - x_s|^{\gamma}}{1+ |t-s|^{\xi}} \ge \frac{1}{2 A_j ^2} | \overline{s}_0 ^{2^{j-1}} - \overline{s}_{2^{j-1}}^{2^{j-1}}|^2.  $$
    This, in turn, already verifies the first claimed statement. A similar argument works to prove the second statement. Indeed, by construction $\Lambda_{\alpha,j-1,T} (\nu ' _j) \preceq \mu^{\times 2}$. But since $\supp(\nu_j ') \subseteq 64 (C_0 \cap \dots \cap C_{j-2})$ it holds true that
    $$\beta \int_0 ^{T} \de t \int_0 ^{T} \de s \frac{| x_t - x_s|^{\gamma}}{1+|t-s|^{\xi}} + \sum\limits_{l= 1} ^{j-1} A_l ^{-2} \sum\limits_{s = 0} ^{T - 2^l} \big| \overline{s}_{s 2^l} ^{2^l} \big|^2$$
    is the uniformly bounded limit of symmetric and quasi-concave functions on path space $\nu_j '$-almost surely by the same calculations as before via Proposition \ref{prop:r_bound}. 
    The second claim then simply follows from the equation
    $$\frac{\de (\Tilde{\nu}'_j)_{\alpha\beta,T,\gamma,\xi}}{\de \mu^{\times 2}} = \frac{\de (\Tilde{\nu}'_j)_{\alpha\beta,T,\gamma,\xi}}{\de \Lambda_{\alpha,j-1,T}(\nu_j ')} \frac{\de \Lambda_{\alpha,j-1,T}(\nu_j ') }{ \de \mu^{\times 2}}. \qedhere$$
\end{proof}

\begin{proof}[Proof of Lemma \ref{lemma:main_variance_inequality}]
Write $f_n(x) := \min \{ x^2,n\}$. Note that $f_n$ is increasing, symmetric, continuous and bounded.
By Lemma \ref{lemma:density_replacement} it is seen that
\begin{equation}
    \nn
    \begin{split}
        \bbE^{(\Tilde{\mu}^m)_{\alpha\beta,T,\gamma,\xi}} [f_n(x_T)] &\le \sum\limits_{j=1} ^t w(j) \bbE^{(\Tilde{\nu}'_j)_{\alpha,T,\gamma,\xi}} [f_n(x_T)]  + w^* \bbE^{\mu_{\alpha ,T} ^\textbf{hier}} [f_n(x_T)]  \\
        &\le \sum\limits_{j=1} ^t w(j) \bbE^{(\mu^m)^{\times 2}} [f_n(x_T)] + w^* \bbE^{\mu_{\alpha,T} ^\textbf{hier}} [f_n(x_T)] \\
        &\le 2T^{-6} + (1-T^{-7}) \bbE^{\mu_{\alpha,T} ^\textbf{hier}} [f_n(x_T)].
    \end{split}
\end{equation}
Here, the first inequality follows by (GCI), and the second by recalling equation \eqref{equ:weight_sum}.
By first taking $m \to \infty$ and then by monotone convergence, we obtain the bound
$$\bbE^{\bbPw_{\alpha\beta,\gamma,\xi,T}} [|x_T|^2] \le (1-T^{-7})\bbE^{\bbP_{\alpha} ^{\textbf{hier}}} [|x_T|^2] + T^{-6}. \qedhere$$
\end{proof}
In order to prove Theorem \ref{thm:main_thm_2} it remains to estimate $\bbE^{\bbP^\mathbf{hier}_{\alpha,T}} [|x_{T}|^2]$. We begin by checking that the variances of the averaged increments are not too large; this is done by studying the recursive relations that are used to define $A_j$ and $r_j$. In a second step we once again rely on Lemma \ref{lemma:variance_decomposition} to bound the endpoint variance by averaged increments, for which we can read the variance directly off the measure.

\begin{proposition}
    \label{prop:bounded_recursion}
    Let $\xi <2$. Then, there exists a constant $c_{\ref{prop:bounded_recursion}} > 0$ such that
    $$r_j \le (\sqrt{20} + c_{\ref{prop:bounded_recursion}})^{2/ \gamma}$$
    for all $j \in \bbN$.
\end{proposition}

\begin{proof}
    Recall that by definition with $\epsilon = \frac{\gamma}{2}\in (0,1)$, $C=256$ and $c = (2- \xi)/2 >0$ that
    $$r_{j+1} = r_j + C r^{1-\epsilon}2^{-c(j+1) }.$$
    Clearly $r_{j+1}\ge r_j$ for all $j \in \bbN$. Take $s_n := r_n ^\epsilon$. By the mean-value Theorem applied to $x \to x^\epsilon$, there exists $\xi \in [r_n,r_{n+1}]$ so that 
    $$s_{n+1} - s_n =  \epsilon  \xi^{\epsilon -1}(r_{n+1}- r_n).$$ 
    By definition $r_{n+1} - r_n = C2^{-cn}r_n^{1-\epsilon}$. As $1-\epsilon < 0$, $\xi^{1-\epsilon} \le r_n^{1-\epsilon}$. Therefore,
    $$s_{n+1} - s_n \le \epsilon C 2^{-cn}.$$
    Using the geometric series and telescoping, we obtain the bound
    $$s_n - s_0 \le \epsilon C \frac{1}{1-2^{-c}}.$$
    Re-arranging and taking roots, we find
    $$r_n \le \lt( 20^{\gamma/2} + \epsilon C \frac{1}{1-2^c}  \rt)^{2/\gamma}.$$
    The right hand side is independent of $n \in \bbN$, which shows the claim.
\end{proof}

\begin{proof}[Proof of Theorem \ref{thm:main_thm_3}]
    The proof follows from Lemma \ref{lemma:variance_decomposition} by noting that $A^*$ is bounded. Therefore, we sum over logarithmically many terms, which shows the claim.
\end{proof}

\begin{appendix}
\section{Closed-form solution for periodic Gaussian measures}
    In this section we show how to find an asymptotic formula for the variance of the path at arbitrary time under a Gaussian path measure. While the formula has been known since at least \cite{Sp87} (see especially formula (3.13)), we could not find a source on how to actually arrive at the result in detail. In the following all calculations are performed and additional assumptions/alterations to the model are mentioned explicitly.

In comparison to the models that we treat in general, one has to add the assumption that $\rho(t) = \rho(T-t)$ for all $t \le T$. If this is apriori not the case, it is possible to add the density
$$\exp \lt(- \alpha \int_0 ^T \int_0 ^T \de t \de s \rho(|t-s-T|)|x_t - x_s|^2 \rt)$$
to symmetrize $\rho$.
Moreover, we cannot take $\bbP$ as Brownian motion started at $0$ and no boundary condition at the end. Instead, we will take the underlying Brownian motion to have periodic boundary conditions. The resulting process is denoted by $\bbP^* _T$ and only defined on the $\sigma$-algebra generated by increments of the path. The path measure we are interested in is given by
\begin{equation}
    \label{equ:explicitly_solvable_model}
\end{equation}
$$\hat{\bbP}^* _{\epsilon,\alpha,T}(\de x) \propto \exp \lt(-\epsilon \int_0 ^T \de s x_s ^2  - \alpha \int_0 ^T \int_0 ^T \de t \de s \rho( |t-s|) (x_t - x_s )^2 \rt) \bbP^* _T (\de x).$$
For $\epsilon = 0$ we abbreviate $\bbPh_{0,\alpha,T} ^* = \bbPh_{\alpha,T} ^*$. This measure is not pinned down, which implies that the variance of $x_s$ for $s \le T$ is unbounded. We fix this by studying the variance of $x_T - x_{T/2}$, which is a translation invariant quantity. It is intuitive to say that
$$\bbE^{\bbPh_{\alpha,T}}[x_T ^2] \approx \bbE^{\hat{\bbP}^*_{\alpha,T}}\lt[(x_T - x_{T/2})^2 \rt]$$
for $T$ large.
We will not rigorously justify this approximation here, but instead postpone it to future work. Similarly, we will first take the limit in volume and then look at the variance at large times; i.e. we let $\lim_{T \to \infty} \hat{\bbP}^* _{\alpha,T} = \hat{\bbP}^* _{\alpha,\infty}$ and then calculate $\bbE^{\hat{\bbP}^* _{\alpha,\infty}} [ x_n ^2]$ for $n >0$ arbitrary. The fact that the joint limit and the iterated limit coincide is again intuitively clear, but will not be shown rigorously in this work. The goal for the remainder of this note is to show the following formula.
\begin{proposition}
    \label{prop:periodic_gaussian_variance}
    Let $m,n >0$. It holds that
    $$\lim\limits_{T \to \infty} \bbE^{\hat{\bbP}^* _{\alpha,T}} [ |x_m - x_{n}|^2] = 2 \int_0 ^\infty \de k \frac{1- \cos \big( 2 \pi (n-m)k \big)}{ 4 \pi^2 k ^2 + \alpha ( \hat{\rho}(0) - \hat{\rho}(k))}.$$
\end{proposition}

Discretizing the measure in \eqref{equ:explicitly_solvable_model}, we find that the variance of $x_m - x_{n}$ is equivalent to calculating the variance of $x_{\Tilde{N}m} - x_{\Tilde{N}n}$ under the measure proportional to
\begin{equation}
    \label{equ:discretization_density}
    \exp\lt(- \frac{\alpha}{\Tilde{N}^2} \sum\limits_{1 \le i < j \le \Tilde{N}T } \rho(|i-j|/\Tilde{N}) \Vert x_{i / \Tilde{N}} - x_{j/\Tilde{N}} \Vert^2 - \frac{\Tilde{N}}{2} \sum\limits_{j=1}^{\Tilde{N}T}\Vert x_{j/\Tilde{N}} - x_{(j-1) \Tilde{N}} \Vert ^2 \rt) \de x
\end{equation}
as $\Tilde{N} \to \infty$. In what follows we abbreviate $N = \Tilde{N}T$ and assume $N$ to be even.
Writing the interaction in matrix form, we find that the above can be expressed as 
$$\e{-\frac{1}{2} \la x,Ax \ra } \de x,$$
where $A$ is given by
$$A = \begin{pmatrix}
    a_0 & a_{N-1} & ... & a_2 & a_1 \\
    a_1 & a_0 & a_{N-1} & ... & a_2 \\
    ... & a_1 & a_0 & ... & ... \\
    a_{N-2} & ... &... & ... & a_{N-1} \\
    a_{N-1} & a_{N-2} &... & ... & a_0 \\
\end{pmatrix}.$$
$A$ is not invertible since the process is not pinned down. We fix this by replacing $a_0$ by $a_0 + \epsilon$, where we will let $\epsilon \to 0$ later on. In our above notation this means that we switch to working with $\bbPh_{\epsilon,\alpha,T} ^*$ and take $\lim_{\epsilon \to 0} \bbPh_{\epsilon,\alpha,T}^*$ subsequently. Taking $\epsilon > 0$ leads to $A$ being invertible without destroying any of the symmetry. We denote the resulting matrix by $A^\epsilon$.

Since $A^\epsilon$ is symmetric, we have the additional symmetry $a_j = a_{N-j}$. Calculating the constants explicitly gives
$$a_0 = \Tilde{N}+ \frac{\alpha}{\Tilde{N}^2}\sum\limits_{l=1} ^N \rho(l/ \Tilde{N}) + \epsilon,$$
$$a_1 = - \frac 1 2 \Tilde{N} - \frac{\alpha}{2\Tilde{N}^2} \rho(1/ \Tilde{N}),$$
$$a_j = \frac{\alpha}{2\Tilde{N}^2} \rho(j/ \Tilde{N})$$
for all $j \le N/2$. All other variables are fixed by symmetry. 
Stated in matrix form, it holds that $$\bbE^{\hat{\bbP}^*_{\alpha,T}}\lt[(x_m - x_{n})^2 \rt] =  \lim\limits_{\Tilde N \to \infty} \lim\limits_{\epsilon \to 0}  (A^\epsilon _{\Tilde{N}m, \Tilde{N}m})^{-1} + (A^\epsilon _{\Tilde{N}n,\Tilde{N}n})^{-1} - 2 (A^\epsilon _{\Tilde{N}m,\Tilde{N}n})^{-1}.$$
One identifies $A$ and therefore $A^\epsilon$ as being circulant, which means that $A^\epsilon$ acts as convolution by the vector
$$a = (...,a_{-1},a_0,a_1,...)$$
with $a_j = a_{j+N}$ for all $j \in \bbZ$. This directly implies that the discrete Fourier transform $\caF_N$ of order $N$ diagonalizes $A^\epsilon$ in the usual sense.
In fact, a simple calculation shows that $\caF_N A = \hat{a} \caF_N$. Here, $\hat{a}$ indicates multiplication by the Fourier modes. By the fact that $A^\epsilon$ is positive definite and symmetric, we find that $\hat{a}(k) > 0$ for all $k \le N$.
It then makes sense to define the multiplication operator $1 / \hat{a}$, which implies $(A^\epsilon)^{-1} = \caF_N ^{-1} \frac{1}{\hat{a}} \caF_N$.

\begin{proposition}
    Let $n,m \le N$. Then, it holds that
    \begin{equation}
        \nn
        \begin{split}
            (A^\epsilon)^{-1}_{n,m} &= \la \delta_n,(A^\epsilon)^{-1} \delta_m \ra \\
            &= \frac{1}{N} \sum\limits_{k \in V}  \frac{\e{i 2 \pi (n-m)k /N}}{ \!\!\epsilon \!+\!\Tilde{N} \!+ \!\frac{\alpha}{\Tilde{N}^2} \sum\limits_{l=1} ^{N} \rho(l/\Tilde{N}) -  \cos(-2 \pi k/N)\Tilde{N} - \frac{\alpha}{\Tilde{N}^2} \sum\limits_{l = 1} ^{N-2} \cos(-2 \pi lk/N) \rho(l/\Tilde{N})} .
        \end{split}
    \end{equation}
    In particular,
    \begin{equation}
        \nn
        \begin{split}
            &\bbE^{\hat{\bbP}^*_{\alpha,T}}\lt[(x_m - x_{n})^2 \rt] \\
            &= \lim\limits_{\Tilde{N} \to \infty} \frac{2}{N} \sum\limits_{k \in V}  \frac{1-\cos(2 \pi (n-m)k /N)}{ \Tilde{N} + \frac{\alpha}{\Tilde{N}^2} \sum\limits_{l=1} ^{N} \rho(l/\Tilde{N}) -  \cos(-2 \pi k/N)\Tilde{N} - \frac{\alpha}{\Tilde{N}^2} \sum\limits_{l = 1} ^{N-2} \cos(-2 \pi lk/N) \rho(l/\Tilde{N})} . 
        \end{split}
    \end{equation}
\end{proposition}

\begin{proof}
    Define the sequence
    $$x_m(k) = \frac{1}{\sqrt{N}}\e{-i2\pi m k/N}.$$
    We find
    \begin{equation}
        \nn
        \begin{split}
            \la \delta_n,(A^\epsilon)^{-1} \delta_m \ra &= \la \delta_n, \caF_N^{-1} \frac{1}{\sqrt{N} \hat{a}} \caF_N \delta_m \ra \\
            &= \la \delta_n, \caF_N^{-1} \frac{1}{\sqrt{N} \hat{a}} x_m \ra.
        \end{split}
    \end{equation}
    By definition
    \begin{equation}
    \label{equ:first_covariance_expression}
        \lt[\caF_N^{-1} \frac{1}{\sqrt{N}\hat{a}} x_m \rt](n) = \frac{1}{N^{3/2}} \sum\limits_{l \in V} \e{i 2 \pi nl /N} \frac{1}{\hat{a}(l)} \e{-i2\pi ml/N} = \frac{1}{N^{3/2}} \sum\limits_{l \in V} \e{i 2 \pi (n-m)l /N} \frac{1}{\hat{a}(l)}.
    \end{equation}
    Inserting the definitions again, we find how $\hat{a}$ relates to $\rho$. In fact,
    \begin{equation}
        \nn
        \begin{split}
            &\sqrt{N} \hat{a}(k) = \sum\limits_{l=0} ^{N-1} \cos(-2 \pi lk/N) a(k) \\
            &=  \Tilde{N} + \frac{\alpha}{\Tilde{N}^2} \sum\limits_{l=1} ^{N} \rho(l/\Tilde{N}) +\epsilon - \cos(- 2 \pi k/N) \lt( \Tilde{N} + \frac{\alpha}{\Tilde{N}^2} \rho(1/\Tilde{N}) \rt)  + \sum\limits_{l = 2} ^{N-2}  \cos(-2 \pi lk/N)a(k)  \\
            &=  \!\!\Tilde{N} \!+\! \frac{\alpha}{\Tilde{N}^2} \sum\limits_{l=1} ^{N} \rho(l/\Tilde{N}) \!+\!\epsilon \!- \! \cos(- 2 \pi k/N) \lt(\!\! \Tilde{N} \!\!+ \!\frac{\alpha}{\!\Tilde{N}^2} \rho(1/\Tilde{N}) \rt) \!- \!\frac{\alpha}{2\Tilde{N}^2} \sum\limits_{l = 2} ^{N-2} \cos(-2 \pi lk/N) \rho(l/\Tilde{N})  \\
            &= \Tilde{N}(1-\cos(- 2 \pi k/N)) + \frac{\alpha}{\Tilde{N}^2} \sum\limits_{l=1} ^{N} \rho(l/\Tilde{N}) +\epsilon - \frac{\alpha}{2\Tilde{N}^2} \sum\limits_{l = 1} ^{N-1} \cos(-2 \pi lk/N) \rho(l/\Tilde{N}).
        \end{split}
    \end{equation}
    Here, the second-to-last equality follows from the fact that the expression
    $$l \mapsto \cos(-2 \pi lk/N) \rho(l/\Tilde{N})$$
    is $-\cos(-2\pi k/N) \lt(\Tilde{N}+\frac{\alpha}{\Tilde{N}^2}\rho(1/\Tilde{N}) \rt)/2$ for $l\in \{1,N-1\}$. Substituting into \eqref{equ:first_covariance_expression} yields
    $$\frac{1}{N} \sum\limits_{k \in V}  \frac{\e{i 2 \pi (n-m)k /N}}{\epsilon +\Tilde{N}(1-\cos(- 2 \pi k/N)) + \frac{\alpha}{\Tilde{N}^2} \sum\limits_{l=1} ^{N} \rho(l/\Tilde{N}) - \frac{\alpha}{2\Tilde{N}^2} \sum\limits_{l = 1} ^{N-1} \cos(-2 \pi lk/N) \rho(l/\Tilde{N})}.$$
    The second claim follows directly by adding the corresponding matrix elements, letting $\epsilon \to 0$ and using symmetry.
    \end{proof}

\begin{proof}[Proof of Proposition \ref{prop:periodic_gaussian_variance}]
An easy calculation shows that 
$$1-\cos(-2\pi k/n) = 4 \pi^2 \frac{k^2}{N^2} + o(N^{-3}).$$
Plugging this into the formula obtained in Proposition \ref{prop:periodic_gaussian_variance} and letting $\Tilde N \to \infty$, we find that (using in the formula $m = \Tilde{N}m$ and the same for $n$ to remove the discretization)
$$\bbE^{\hat{\bbP}^* _{\alpha,T}} [ (x_m - x_n)^2] = \frac{2}{T}\sum\limits_{k=1}^\infty \frac{1 - \cos(2 \pi k (m-n) /T)}{4 \pi^2 (k/T)^2 + \alpha \int_0 ^T \de x \rho(x) - \alpha \int_0 ^T \de l \e{- 2 \pi i l k/T} \rho(l)}.$$
Note that all variables we sum over are normalized by $T$. Taking $T \to \infty$ yields the claim.
\end{proof}

\end{appendix}

\bigskip
{\bf Acknowledgments:} This research was supported by DFG grant No 535662048.
\bibliographystyle{plain} 
\bibliography{bib}

@article{We87,
  title={On scale mixtures of normal distributions},
  author={West, Mike},
  journal={Biometrika},
  volume={74},
  number={3},
  pages={646--648},
  year={1987},
  publisher={Oxford University Press}
}

@article{Se24surface,
  title={{Localization of Random Surfaces with Monotone Potentials and an FKG-Gaussian Correlation Inequality}},
  author={Sellke, Mark},
  journal={arXiv preprint arXiv:2402.18737},
  year={2024}
}

@article{Se24,
  title={{Almost quartic lower bound for the Fr{\"o}hlich polaron’s effective mass via Gaussian domination}},
  author={Sellke, Mark},
  journal={Duke Mathematical Journal},
  volume={173},
  number={13},
  pages={2687--2727},
  year={2024},
  publisher={Duke University Press}
}

@article{Sch25,
title = {{Explicit bounds for a Gaussian decomposition Lemma of Sellke}},
journal = {Statistics and Probability Letters},
volume = {219},
pages = {110342},
year = {2025},
issn = {0167-7152},
doi = {https://doi.org/10.1016/j.spl.2024.110342},
url = {https://www.sciencedirect.com/science/article/pii/S0167715224003110},
author = {Tobias Schmidt}
}

@article{Ro14,
title = {{A simple proof of the Gaussian correlation conjecture extended to some
multivariate gamma distributions}},
journal = {Far East J. Theor. Stat.},
volume = {48},
pages = {139-145},
year = {2014},
issn = {},
doi = {},
url = {},
author = {Thomas Royen}
}

@ARTICLE{Be03,
    AUTHOR = {Volker Betz},
     TITLE = {Existence of Gibbs Measures Relative to Brownian Motion},
   JOURNAL = {Markov Processes And Related Fields},
  FJOURNAL = {Markov Processes And Related Fields},
      YEAR = {2003},
    VOLUME = {9},
       PNO = {1},
     PAGES = {85-102}
}

@ARTICLE{BeSchSe25,
    AUTHOR = {Volker Betz and Tobias Schmidt and Mark Sellke},
     TITLE = {Mean square displacement of Brownian paths perturbed by bounded pair potentials},
   JOURNAL = {Electron. J. Probab.},
  FJOURNAL = {Electronic Journal of Probability},
      YEAR = {2025},
    VOLUME = {30},
       PNO = {3},
     PAGES = {1-17},
      ISSN = {1083-6489},
       DOI = {10.1214/24-EJP1263},
      SICI = {1083-6489(2025)30:3<1:MSDOBP>2.0.CO;2-H},
}

@article{BeSchSe24,
  title={Enhanced binding for a quantum particle coupled to scalar quantized field},
  author={Betz, Volker and Schmidt, Tobias and Sellke, Mark},
  journal={arXiv preprint arXiv:2410.22569},
  year={2024}
}

@book{HiLo20,
url = {https://doi.org/10.1515/9783110403541},
title = {Volume 2 Applications in Rigorous Quantum Field Theory},
author = {Fumio Hiroshima and József Lörinczi},
publisher = {De Gruyter},
address = {Berlin, Boston},
doi = {doi:10.1515/9783110403541},
isbn = {9783110403541},
year = {2020},
lastchecked = {2025-01-13}
}

@article{BePo22,
  title={A functional central limit theorem for polaron path measures},
  author={Betz, Volker and Polzer, Steffen},
  journal={Communications on Pure and Applied Mathematics},
  volume={75},
  number={11},
  pages={2345--2392},
  year={2022},
  publisher={Wiley Online Library}
}

@article{BePo23,
  title={{Effective mass of the Polaron: a lower bound}},
  author={Betz, Volker and Polzer, Steffen},
  journal={Communications in Mathematical Physics},
  volume={399},
  number={1},
  pages={173--188},
  year={2023},
  publisher={Springer}
}

@article{MuVa19,
author = {Mukherjee, Chiranjib and Varadhan, S. R. S.},
year = {2019},
month = {08},
pages = {},
title = {{Identification of the Polaron Measure I: Fixed Coupling Regime and the Central Limit Theorem for Large Times}},
volume = {73},
journal = {Communications on Pure and Applied Mathematics},
doi = {10.1002/cpa.21858}
}

@article{BaMuSeVa23,
  title={{Effective mass of the Fröhlich Polaron and the Landau-Pekar-Spohn conjecture}},
  author={Bazaes, Rodrigo and Mukherjee, Chiranjib and Sellke, Mark and Varadhan, S. R. S.},
  journal={arXiv preprint arXiv:2307.13058},
  year={2023}
}

@article{OsaSp99,
title = "Gibbs measures relative to Brownian motion",
author = "Hirofumi Osada and Herbert Spohn",
year = "1999",
month = jul,
doi = "10.1214/aop/1022677444",
language = "English",
volume = "27",
pages = "1183--1207",
journal = "Annals of Probability",
issn = "0091-1798",
publisher = "Institute of Mathematical Statistics",
number = "3",
}

@article{Sp87,
title = {Effective mass of the polaron: A functional integral approach},
journal = {Annals of Physics},
volume = {175},
number = {2},
pages = {278-318},
year = {1987},
issn = {0003-4916},
doi = {https://doi.org/10.1016/0003-4916(87)90211-9},
url = {https://www.sciencedirect.com/science/article/pii/0003491687902119},
author = {Herbert Spohn}
}

@article{Ba72,
  title = {Ising Model with a Scaling Interaction},
  author = {Baker, George A.},
  journal = {Phys. Rev. B},
  volume = {5},
  issue = {7},
  pages = {2622--2633},
  numpages = {0},
  year = {1972},
  month = {Apr},
  publisher = {American Physical Society},
  doi = {10.1103/PhysRevB.5.2622},
  url = {https://link.aps.org/doi/10.1103/PhysRevB.5.2622}
}

@article{BeSp05,
  title={{A central limit theorem for Gibbs measures relative to Brownian motion}},
  author={Betz, Volker and Spohn, Herbert},
  journal={Probability Theory and Related Fields},
  volume={131},
  number={3},
  pages={459--478},
  year={2005},
  publisher={Springer}
}

@article{Sp86,
doi = {10.1088/0305-4470/19/4/014},
url = {https://dx.doi.org/10.1088/0305-4470/19/4/014},
year = {1986},
month = {mar},
publisher = {},
volume = {19},
number = {4},
pages = {533},
author = {H Spohn},
title = {Roughening and pinning transitions for the polaron},
journal = {Journal of Physics A: Mathematical and General}
}

@article{BoSch97,
author = {Erwin Bolthausen and Uwe Schmock},
title = {{On self-attracting $d$-dimensional random walks}},
volume = {25},
journal = {The Annals of Probability},
number = {2},
publisher = {Institute of Mathematical Statistics},
pages = {531 -- 572},
year = {1997},
doi = {10.1214/aop/1024404411},
URL = {https://doi.org/10.1214/aop/1024404411}
}

@article{BrySla95, title={The diffusive phase of a model of self-interacting walks}, volume={103}, ISSN={1432-2064}, DOI={10.1007/BF01195476}, abstractNote={We consider simple random walk onZd perturbed by a factor exp[βT−PJT], whereT is the length of the walk and$$J_T  = sumnolimits_{0 leqslant i< j leqslant T} delta  _{omega (i),omega (j)} $$. Forp=1 and dimensionsd≥2, we prove that this walk behaves diffusively for all − ∞ < β <0, with β0 > 0. Ford>2 the diffusion constant is equal to 1, but ford=2 it is renormalized. Ford=1 andp=3/2, we prove diffusion for all real β (positive or negative). Ford>2 the scaling limit is Brownian motion, but ford≤2 it is the Edwards model (with the “wrong” sign of the coupling when β>0) which governs the limiting behaviour; the latter arises since for$$p = frac{{4 - d}}{2}$$,T−pJTis the discrete self-intersection local time. This establishes existence of a diffusive phase for this model. Existence of a collapsed phase for a very closely related model has been proven in work of Bolthausen and Schmock.}, number={3}, journal={Probability Theory and Related Fields}, author={Brydges, D. C. and Slade, G.}, year={1995}, month=sep, pages={285–315} }

@article{Bo94,
 ISSN = {00911798, 2168894X},
 URL = {http://www.jstor.org/stable/2244898},
 author = {Erwin Bolthausen},
 journal = {The Annals of Probability},
 number = {2},
 pages = {875--918},
 publisher = {Institute of Mathematical Statistics},
 title = {Localization of a Two-Dimensional Random Walk with an Attractive Path Interaction},
 urldate = {2025-07-16},
 volume = {22},
 year = {1994}
}

@article{DoVa83,
author = {Donsker, M. D. and Varadhan, S. R. S.},
title = {Asymptotics for the polaron},
journal = {Communications on Pure and Applied Mathematics},
volume = {36},
number = {4},
pages = {505-528},
doi = {https://doi.org/10.1002/cpa.3160360408},
url = {https://onlinelibrary.wiley.com/doi/abs/10.1002/cpa.3160360408},
eprint = {https://onlinelibrary.wiley.com/doi/pdf/10.1002/cpa.3160360408},
year = {1983}
}

@article{BoDeuSch93,
         journal = {Probability Theory and Related Fields},
          number = {3},
           pages = {283--310},
          author = {E Bolthausen and J-D Deuschel and U Schmock},
       publisher = {Springer},
          volume = {95},
            year = {1993},
           title = {Convergence of path measures arising from a mean field or polaron type interaction},
            issn = {0178-8051},
        language = {english},
             doi = {10.1007/BF01192166},
             url = {https://doi.org/10.5167/uzh-22673}
}

@article{BeHiKrPo25,
  title={{On the Ising Phase Transition in the Infrared-Divergent Spin Boson Model}},
  author={Betz, Volker and Hinrichs, Benjamin and Kraft, Mino Nicola and Polzer, Steffen},
  journal={arXiv preprint arXiv:2501.19362},
  year={2025}
}

@article{Dy69, title={Existence of a phase-transition in a one-dimensional Ising ferromagnet}, volume={12}, ISSN={1432-0916}, DOI={10.1007/BF01645907}, abstractNote={Existence of a phase-transition is proved for an infinite linear chain of spins μj=±1, with an interaction energy$$H =  - sum J(i - j)mu _i mu _j ,$$whereJ(n) is positive and monotone decreasing, and the sums ΣJ(n) and Σ (log logn) [n3J(n)]−1 both converge. In particular, as conjectured byKac andThompson, a transition exists forJ(n)=n−α when 1 < α < 2. A possible extension of these results to Heisenberg ferromagnets is discussed.}, number={2}, journal={Communications in Mathematical Physics}, author={Dyson, Freeman J.}, year={1969}, month=jun, pages={91–107} }

@book{BauSlaGor19,
  title={Introduction to a renormalisation group method},
  author={Bauerschmidt, Roland and Brydges, David C and Slade, Gordon},
  volume={2242},
  year={2019},
  publisher={Springer Nature}
}

@article{Ne64,
    author = {Nelson, Edward},
    title = {Interaction of Nonrelativistic Particles with a Quantized Scalar Field},
    journal = {Journal of Mathematical Physics},
    volume = {5},
    number = {9},
    pages = {1190-1197},
    year = {1964},
    month = {09},
    issn = {0022-2488},
    doi = {10.1063/1.1704225},
    url = {https://doi.org/10.1063/1.1704225},
    eprint = {https://pubs.aip.org/aip/jmp/article-pdf/5/9/1190/19309356/1190\_1\_online.pdf},
}

@article{HiMa22,
  title={{Ground states and associated path measures in the renormalized Nelson model}},
  author={Hiroshima, Fumio and Matte, Oliver},
  journal={Reviews in Mathematical Physics},
  volume={34},
  number={02},
  pages={2250002},
  year={2022},
  publisher={World Scientific}
}

@article{BrSe24,
  title={{The Fr{\"o}hlich Polaron at strong coupling: part II—energy-momentum relation and effective mass}},
  author={Brooks, Morris and Seiringer, Robert},
  journal={Publications math{\'e}matiques de l'IH{\'E}S},
  volume={140},
  number={1},
  pages={271--309},
  year={2024},
  publisher={Springer}
}

@article{BrSe23,
  title={{The Fr{\"o}hlich polaron at strong coupling: Part I—The quantum correction to the classical energy}},
  author={Brooks, Morris and Seiringer, Robert},
  journal={Communications in Mathematical Physics},
  volume={404},
  number={1},
  pages={287--337},
  year={2023},
  publisher={Springer}
}

@article{Br24,
  title={Proof of the Landau-Pekar Formula for the effective Mass of the Polaron at strong coupling},
  author={Brooks, Morris},
  journal={arXiv preprint arXiv:2409.08835},
  year={2024}
}

@inproceedings{DySp20,
  title={Effective mass of the polaron—revisited},
  author={Dybalski, Wojciech and Spohn, Herbert},
  booktitle={Annales Henri Poincar{\'e}},
  volume={21},
  number={5},
  pages={1573--1594},
  year={2020},
  organization={Springer}
}

@article{HaHiSi24,
  title={{Non-Fock ground states in the translation-invariant Nelson model revisited non-perturbatively}},
  author={Hasler, David and Hinrichs, Benjamin and Siebert, Oliver},
  journal={Journal of Functional Analysis},
  volume={286},
  number={7},
  pages={110319},
  year={2024},
  publisher={Elsevier}
}

@article{HiSp01,
  author       = {Hiroshima, F. and Spohn, H.},
  title        = {Enhanced Binding Through Coupling to a Quantum Field},
  journal      = {Annales Henri Poincaré},
  volume       = {2},
  number       = {6},
  pages        = {1159--1187},
  year         = {2001},
  doi          = {10.1007/s00023-001-8606-1},
  url          = {https://doi.org/10.1007/s00023-001-8606-1},
  issn         = {1424-0637}
}

@article{Fr73,
    author = {Fr{\"o}hlich, J{\"u}rg},
    title = "{On the infrared problem in a model of scalar electrons and massless, scalar bosons}",
    journal = "Ann. Inst. H. Poincare Phys. Theor. A",
    volume = "19",
    number = "1",
    pages = "1--103",
    year = "1973"
}

@article{Fr74,
author = {Fröhlich, Jürg},
title = {Existence of Dressed One Electron States in a Class of Persistent Models},
journal = {Fortschritte der Physik},
volume = {22},
number = {3},
pages = {159-198},
doi = {https://doi.org/10.1002/prop.19740220304},
url = {https://onlinelibrary.wiley.com/doi/abs/10.1002/prop.19740220304},
eprint = {https://onlinelibrary.wiley.com/doi/pdf/10.1002/prop.19740220304},
year = {1974}
}

@ARTICLE{HVV03,
  title     = "Enhanced binding in non-relativistic {QED}",
  author    = "Hainzl, Christian and Vougalter, Vitali and Vugalter, Semjon A",
  journal   = "Commun. Math. Phys.",
  publisher = "Springer Science and Business Media LLC",
  volume    =  233,
  number    =  1,
  pages     = "13--26",
  month     =  feb,
  year      =  2003
}

@ARTICLE{HiSa08,
  title     = "Enhanced binding of an N-particle system interacting with a
               scalar bose field {I}",
  author    = "Hiroshima, Fumio and Sasaki, Itaru",
  journal   = "Math. Z.",
  publisher = "Springer Science and Business Media LLC",
  volume    =  259,
  number    =  3,
  pages     = "657--680",
  month     =  jul,
  year      =  2008,
  language  = "en"
}

@ARTICLE{KoMa13,
  title     = "On enhanced binding and related effects in the non- and
               semi-relativistic Pauli-fierz models",
  author    = "K{\"o}nenberg, Martin and Matte, Oliver",
  journal   = "Commun. Math. Phys.",
  publisher = "Springer Science and Business Media LLC",
  volume    =  323,
  number    =  2,
  pages     = "635--661",
  month     =  oct,
  year      =  2013,
  language  = "en"
}

@ARTICLE{AKS25,
  title     = "Brownian motion conditioned to spend limited time outside a
               bounded interval -- an extreme example of entropic repulsion",
  author    = "Aurzada, Frank and Kolb, Martin and Schickentanz, Dominic T",
  journal   = "Bernoulli (Andover.)",
  publisher = "Bernoulli Society for Mathematical Statistics and Probability",
  volume    =  31,
  number    =  2,
  month     =  may,
  year      =  2025
}
\end{document}